\documentclass[11pt]{article}
\usepackage[utf8]{inputenc}
\usepackage[usenames,dvipsnames,svgnames,table]{xcolor}	
\usepackage[]{amsmath}
\usepackage{cases}
\usepackage{empheq}
\usepackage[colorlinks=true]{hyperref}
\hypersetup{colorlinks,citecolor=brown, linkcolor=blue, linktocpage}
\usepackage{cleveref} 
\usepackage{hyperref}
\usepackage{amssymb}
\usepackage{bbm}
\usepackage{enumerate}   
\usepackage{geometry}
\usepackage{amsfonts}
\usepackage{color}
\usepackage[]{amsthm} 
\usepackage[shortlabels]{enumitem}

\newtheorem{theorem}{Theorem}
\newtheorem{lemma}[theorem]{Lemma}
\newtheorem{cor}[theorem]{Corollary}
\newtheorem{proposition}[theorem]{Proposition}
\newtheorem{mydef}[theorem]{Definition}

\newtheorem{remark}[theorem]{Remark}

\def\E{\mathbb{E}}
\def\R{\mathbb{R}}
\def\P{\mathbb{P}}
\def\Q{\mathbb{Q}}

\def\1{\mathbbm{1}}

\def\D1{\frac{\partial}{\partial x}}
\newcommand{\tX}{\tilde X}
\newcommand{\tW}{\tilde W}
\newcommand{\dd}{{\rm d}}

\newcommand{\Ff}{\mathcal F}
\newcommand{\Gg}{\mathcal G}

\newcommand{\Kk}{{\mathcal K}}

\newcommand{\Pp}{\mathcal P}

\newcommand{\e}{\varepsilon}

\usepackage{xspace}

\title{Particle approximation of the doubly parabolic Keller-Segel equation in the plane}

\author{Nicolas Fournier\footnote{Sorbonne Université and Université Paris Cité, CNRS, Laboratoire de Probabilités, Statistique et Modélisation, F-75005 Paris, France. nicolas.fournier@sorbonne-universite.fr} \hskip2mm  and Milica Toma\v sevi\'c\footnote{CMAP, CNRS, École polytechnique, Institut Polytechnique de
Paris, 91120 Palaiseau, France. milica.tomasevic@polytechnique.edu (corresponding author)} }

\date{}
\begin{document}

\maketitle

\begin{abstract}
In this work, we study a stochastic system of $N$ particles  associated with the parabolic-parabolic Keller-Segel system. This particle system is singular and non Markovian in that its drift term depends on the past of the particles. When the sensitivity parameter is sufficiently small, we show that this particle system indeed exists for any $N\geq 2$, we show  tightness in $N$ of its empirical measure, and that any weak limit point of this empirical measure,
as $N\to\infty$, solves some nonlinear martingale problem, which in particular implies that its family of time-marginals solves the parabolic-parabolic Keller-Segel system in some weak sense.
The main argument of the proof consists of a \textit{Markovianization} of the interaction kernel: We show that, in some loose sense, the two-by-two path-dependant interaction can be controlled by a two-by-two Coulomb interaction, as in the parabolic-elliptic case.

\end{abstract}

\section{Introduction and main results}\label{inin}

In this work, we study a stochastic particle approximation of the parabolic-parabolic Keller-Segel equation for chemotaxis in the plane. This equation, with unknown $(\rho,c)$, writes
\begin{equation}
\label{EDP_KS}
\begin{cases}
 &\partial_t \rho_t(x) = \Delta  \rho_t(x)-\chi \nabla \cdot (\rho_t(x)\nabla c_t(x)), \hskip1.5cm  t>0,\quad x\in\R^2, \\[4pt]
 &\theta \partial_t c_t(x) = \Delta c_t(x)-\lambda c_t(x)+\rho_t(x), \hskip2.15cm t>0, \quad x\in\R^2,
\end{cases}
\end{equation}
with $\rho_0$ and $c_0$ given.
Here,  $\rho_t\geq 0$ represents the distribution at time $t\geq 0$ of a cell population. These cells are attracted by a chemical substance that they emit (the so-called chemo-attractant) and whose  concentration at time $t\geq 0$ is given by $c_t\geq 0$. The parameters $\chi>0$, $\theta>0$ and $\lambda\geq 0$ respectively stand for the sensitivity of cells to the chemo-attractant,
the ratio between the diffusion time scales of bacteria and chemo-attractant, and the death rate  of the chemo-attractant.  All along this work, we will suppose the system is rescaled so that the total mass of the cell population (which is preserved in time) is equal to 1. 
We refer to the original works of Keller-Segel \cite{KELLER_SEGEL1, KELLER_SEGEL2, KELLER_SEGEL3} for the initial motivation and some  biological explanations and to the paper of Horstmann \cite{Horstmann1} for a thorough review.

\medskip

An interesting feature of the Keller-Segel system is that its solutions may blow-up in finite time, although the total mass is preserved: Some point cluster may emerge
due to the attraction between cells (through the chemo-attractant).
Namely, it is well known that
for any reasonable initial condition, the parabolic-elliptic version of the system,
which corresponds to the case $\theta=0$, explodes in finite time if $\chi>8\pi$,
while the global well-posedness holds when $\chi\leq 8\pi$. This was first rigorously established by Blanchet-Dolbeault-Perthame \cite{bdp} ($\chi<8\pi$) and Biler-Karch-Lauren\c{c}ot-Nadzieja
\cite{bkln} ($\chi\leq 8\pi$, radial case). In the parabolic-parabolic case where $\theta>0$, the  global well-posedness still holds for $\chi<8\pi$ and any reasonable condition, see Calvez-Corrias \cite{CalCor2008}. However, in the special case of $c_0= 0$ and for any $\chi>0$,
the global well-posedness holds true when $\theta>0$ is large enough, see Biler-Guerra-Karch \cite{Biler-Guerra-Karch}, and this was extend to a more general class of initial concentrations $c_0$ (with a smallness condition depending on $\theta$) by  Corrias-Escobedo-Matos \cite{Corrias2014}. 
Concerning explosion, the situation is still largely open, let us mention that radial solutions on a disk in $\R^2$ blowing-up for $\chi>8\pi$  have been  exhibited by Herrero-Velasquez \cite{HerreroandVelazquez}.  In addition, a criterion for explosion  of radial solutions has been obtained by Mizoguchi \cite{mizoguchiBlowUp}: The conditions
are that $\chi>8\pi$ and that some energy of the initial condition $(\rho_0,c_0)$
is large enough.

\medskip

Our goal is to derive the system \eqref{EDP_KS} as a mean-field limit of an interacting particle system. To this end, we adopt the decoupling strategy proposed by Talay-Toma\v sevi\'c \cite{Mi-De} in order to obtain a stochastic particle description of the system. Namely, observe that the concentration of chemo-attractant can be made explicit in terms of $c_0$ and of the density of bacteria: Using the Duhamel formula, we have
$$
c_t(x)=b_t^{c_0,\theta,\lambda}(x)+ \int_0^t  (K^{\theta,\lambda}_{t-s} \ast \rho_s)(x) \dd s,
$$ 
where we denoted, for $ (t,x) \in (0,\infty)\times \R^2$,
\begin{equation}
g^\theta_t(x):=\frac{\theta}{4\pi t} e^{-\frac{\theta}{4t}|x|^2}, \quad K_t^{\theta,\lambda}(x):= \frac{1}{\theta} e^{-\frac{\lambda}{\theta}t}g^\theta_t(x),  \quad b_t^{c_0,\theta,\lambda} (x):= e^{-\frac{\lambda}{\theta}t}(g^\theta_t\ast  c_0)(x).
\end{equation}
Then, the cell density in \eqref{EDP_KS} is interpreted as a Fokker-Planck equation of a non-linear  stochastic process where in the place of $\nabla c$, the gradient of the above formulation is plugged in. 
Hence, the corresponding non-linear S.D.E. reads
\begin{equation}\label{eq:NLSDE}
    X_t= X_0 + \sqrt{2} W_t + \chi \int_0^t \nabla b_s^{c_0,\theta,\lambda}(X_s)\dd s + \chi \int_0^t\int_0^s (\nabla K^{\theta,\lambda}_{s-u}\ast \rho_u)(X_s) \dd u\dd s,
    \quad \rho_s={\rm Law}(X_s),
\end{equation}
where the $\rho_0$-distributed random variable $X_0$ and the $2D$-Brownian motion $(W_t)_{t\geq 0}$ are independent.
Intuitively, it represents the motion of a typical cell in an infinite cloud of cells undergoing the dynamics of \eqref{EDP_KS}. Then, for $N\geq 2$, the following particle system is its microscopic counterpart:
$$  X_t^{i,N}= X_0^{i,N} + \sqrt{2} W_t^i + \chi \int_0^t \nabla b_s^{c_0,\theta,\lambda}(s,X_s^{i,N})\dd s 
+ \frac{\chi}{N-1} \sum_{j\neq i}\int_0^t \int_0^s \nabla K^{\theta,\lambda}_{s-u} ( X^{i,N}_s-X^{j,N}_u)  \dd u \dd s,
$$
where the initial condition $(X_0^{i,N})_{i=1,\dots,N}$ is independent of the i.i.d. family
$((W^{i}_t)_{t\geq 0})_{i=1,\dots,N}$ of $2D$-Brownian motions.
Noting that 
$$\nabla K_t^{\theta,\lambda}(x)=-\frac{\theta}{8\pi t^2} e^{-\frac{\lambda}{\theta}t } e^{- \frac{\theta}{4t}|x|^2} x,  $$
we point out here that it is not at all clear that this system is well-defined. Indeed, each particle interacts with the other particles by means of a singular functional of their trajectories. 

\medskip

Let us explain quickly what we mean by {\it singular}: Assume that particles do encounter,
as is the case in the parabolic-elliptic case, see \cite{FournierTardy}. If $X^{1,N}_t=X^{2,N}_t$ at some time $t>0$, we may expect that $|X^{1,N}_t-X^{2,N}_s|=|X^{1,N}_t-X^{2,N}_t|
+|X^{2,N}_t-X^{2,N}_s|\simeq (t-s)^{\alpha}$, for some $\alpha \in (0,1)$. Hence
the corresponding interaction (in the drift of $X^{1,N}$) is of order,
$$
\int_{0}^t |\nabla K^{\theta,\lambda}_{t-s} (X^{1,N}_t-X^{2,N}_s)|
\dd s\simeq  \int_{0}^t \frac {(t-s)^{\alpha}}{(t-s)^2}
e^{-\frac{\theta}{4(t-s)}{(t-s)^{2\alpha}}} e^{- \frac \lambda \theta (t-s)} \dd s.
$$
This quantity diverges if and only if $\alpha\geq 1/2$. 
Since we precisely expect the paths of the 
particles to be slightly more irregular than H\"older$(\frac12)$-continuous (as the Brownian motion), it is not clear whether the drift is well-defined or not, and we are really around the critical exponent.

\medskip

Our objective is to show that such a particle system exists and to prove the convergence of its empirical measure,
as $N\to \infty$ and up to extraction of a subsequence, towards a solution to $\eqref{EDP_KS}$, under an explicit (though complicated) smallness condition on the parameter $\chi$.  To the best of our knowledge, this is the first time the doubly parabolic Keller-Segel system on the plane is derived   as a mean field limit of a non-smoothed interacting stochastic particle system. 

\medskip

Let us start with defining our notion of solution to the system \eqref{EDP_KS}. 
We endow the set $\Pp(\R^2)$ of probability measures on $\R^2$ with the weak convergence topology (i.e. with continuous and bounded test functions).
We denote by $C^k_b(\R^2)$ the set of $C^k$ functions on $\R^2$ with bounded derivatives of order $0$ to $k$.

\begin{mydef}
\label{def:soulutionPDE}
Fix $\rho_0 \in \Pp(\R^2)$ and some nonnegative $c_0\in L^p(\R^2)$
for some $p>2$.
A couple $(\rho_t,c_t)_{t\geq 0}$ is a weak solution to \eqref{EDP_KS} if $(\rho_t)_{t\geq 0} \in C([0,\infty),\Pp(\R^2))$, if
for all $t\geq 0$,
\begin{equation}
        \label{eq:ConditionDef}
        \int_0^t \int_{\R^2} \int_0^s \int_{\R^2}   (K^{\theta,\lambda}_{s-u}(x-y) + |\nabla K^{\theta,\lambda}_{s-u}(x-y)|) \rho_u(\dd y)  \dd u   \rho_s(\dd x) \dd s < \infty,
\end{equation}
if for $\rho_t(\dd x) \dd t$-almost every $(t,x)\in \R_+\times \R^2$,
\begin{equation}
        \label{eq:C}
        c_t(x)=  b_t^{c_0,\theta,\lambda}(x)+ \int_0^t  (K^{\theta,\lambda}_{t-s} \ast \rho_s)(x) \dd s,
\end{equation}
and if for all $\varphi\in C^2_b(\R^2)$ and all $t>0$,
\begin{align}
\label{eq:faible}
\int_{\R^2}\!\varphi(x) \rho_t(\dd x) = & \int_{\R^2}\!\varphi(x) \rho_0(\dd x)  + \int_0^t \! \int_{\R^2}\!\Delta \varphi(x) \rho_s(\dd x) \dd s + \chi \int_0^t \int_{\R^2} \nabla \varphi(x) \cdot \nabla c_s(x) \rho_s(\dd x) \dd s.
\end{align}
\end{mydef}

The fact that $c_0 \in L^p(\R^2)$ implies that $b^{c_0,\theta,\lambda}_t$ and $\nabla b^{c_0,\theta,\lambda}_t$ are continuous on $\R^d$ for each $t>0$ and that there is a constant $A=A(\theta,p)$ such that for all $t>0$, with $p'=\frac p{p-1}$,
\begin{equation}\label{estib}
\begin{cases}
\sup_{x \in \R^2} |b^{c_0,\theta,\lambda}_t(x)| \leq ||g^\theta_t||_{L^{p'}}||c_0||_{L^p}\leq \frac{A }{t^{\frac1p}}||c_0||_{L^p},\\[10pt] 
\sup_{x \in \R^2} |\nabla b^{c_0,\theta,\lambda}_t(x)| \leq ||\nabla g^\theta_t||_{L^{p'}}||c_0||_{L^p}\leq \frac{A }{t^{\frac12+\frac1p}}||c_0||_{L^p}.
\end{cases}
\end{equation}
Since $p>2$, \eqref{estib} and \eqref{eq:ConditionDef} imply that
$\int_0^t (c_s(x)+|\nabla c_s(x)|)\rho_s(\dd x)\dd s <\infty$ for all $t>0$
and that $\nabla c_t(x)=\nabla b_t^{c_0,\theta,\lambda}(x)+ \int_0^t  (\nabla K^{\theta,\lambda}_{t-s} \ast \rho_s)(x) \dd s$ for $\rho_t(\dd x) \dd t$-almost every $(t,x)\in \R_+\times \R^2$.
Thus everything makes sense in \eqref{eq:faible}.

\medskip

We now introduce the martingale problem characterizing the law of the nonlinear S.D.E.~\eqref{eq:NLSDE}. 

\begin{mydef}
\label{defMP}
Fix some $\rho_0 \in \Pp(\R^2)$ and some nonnegative $c_0\in L^p(\R^2)$
for some $p>2$. Consider the canonical space $C([0,\infty),\R^2)$ equipped with its
canonical process $(w_t)_{t\geq 0}$ and its canonical filtration. Let $\Q$ be a probability measure on this canonical space and denote $(\Q_t)_{t\geq 0}$ its family of one-dimensional time marginals. We say that $\Q$ solves the non-linear 
martingale problem \hyperref[defMP]{$\mathcal{(MP)}$} with initial law $\rho_0$ if $\Q_0 = \rho_0$, if 
\begin{equation}
    \label{eq:ConditionMP}
    \int_0^t \int_{\R^2} \int_0^s \int_{\R^2}   (K^{\theta,\lambda}_{s-u}(x-y) + |\nabla K^{\theta,\lambda}_{s-u}(x-y)|) \Q_u(\dd y)  \dd u   \Q_s(\dd x) \dd s < \infty,
\end{equation}
and if for any $\varphi \in C_c^2(\R^2)$, the process
\begin{equation}
\label{def_mart}
M^\varphi_t:=\varphi(w_t)-f(w_0)-\int_0^t \Big[ \Delta \varphi(w_u)+\chi \nabla \varphi(w_u)\cdot \Big(  \nabla b_s^{c_0,\theta,\lambda}(w_s) + \int_0^u (\nabla K^{\theta,\lambda}_{u-r}\ast \Q_r)(w_u) \dd r \Big) \Big]\dd u
\end{equation}
is a $\Q$-martingale.
\end{mydef}

For $\Q$ a solution to \hyperref[defMP]{$\mathcal{(MP)}$},
setting $c_t=b_t^{c_0,\theta,\lambda}+ \int_0^t (K^{\theta,\lambda}_{t-s} \ast \Q_s )\,\dd s$, it holds that 
$(\Q_t,c_t)_{t\geq 0}$ is a weak solution to \eqref{EDP_KS}. 
Finally, we consider the following notion of solution to our particle system. 

\begin{mydef}\label{def:PS} Fix $N\geq 2$ and some nonnegative
$c_0\in L^p(\R^2)$ for some $p>2$. Consider some i.i.d. family $(W_t^i)_{t\geq 0,i=1,\dots,N}$ of $2D$-Brownian motion, as well as some exchangeable family 
$(X_0^i)_{i=1,\dots,N}$ of $\R^2$-valued random variables, independent of the family of Brownian motions. A family of continuous $\R^2$-valued processes $(X^{i,N}_t)_{t\geq 0,i=1,\dots,N}$ is said to be a $N$-Keller-Segel particle system if a.s., for all $t\geq 0$, all $i\neq j$,
\begin{align}
    \label{PScond}
    \int_0^t \int_0^s |\nabla K^{\theta,\lambda}_{s-u} ( X^{i,N}_s-X^{i,N}_u)|  \dd u \dd s<\infty 
\end{align}
and if a.s., for all $t\geq 0$, all $i=1,\dots, N$,
\begin{align*}
    X_t^{i,N}=& X_0^{i,N} + \sqrt{2} W_t^i + \chi \int_0^t \nabla b_s^{c_0,\theta,\lambda}(s,X_s^{i,N})\dd s 
    + \frac{\chi}{N-1} \sum_{j\neq i}\int_0^t \int_0^s \nabla K^{\theta,\lambda}_{s-u} ( X^{i,N}_s-X^{j,N}_u)  \dd u \dd s. 
\end{align*} 
\end{mydef}
Everything makes sense in this last expression by \eqref{estib} and \eqref{PScond}.
As already mentioned, it is not at all clear that the above system has a solution and even less that it converges as $N\to \infty$, due to the singular nature of the path-dependent interaction of the particles.
We are ready to present our main result. It gathers several statements that we make throughout the paper. 
\begin{theorem}
\label{th:mainTH}
Let $\chi>0$, $\lambda\geq 0$ and $\theta>0$. Consider $\rho_0 \in \Pp(\R^2)$ and some nonnegative $c_0\in L^p(\R^2)$, for some $p>2$. Consider, for each $N\geq 2$, some exchangeable initial condition $(X^{i,N}_0)_{i=1,\dots,N}$. Suppose that $\chi < \chi_{\theta,p}^*$, where $\chi_{\theta,p}^*>0$ is defined in~\eqref{Eq:chistar}.  Then, we have the following results.
\medskip

(i)  For each $N\geq 2$, there exists an exchangeable $N$-Keller-Segel particle system.

\medskip    

(ii) We set $\mu^N:=\frac1N\sum_{i=1}^N \delta_{(X^{i,N}_t)_{t\geq 0}}$, which a.s. belongs to $\Pp(C([0,\infty),\R^2))$ and, for each $t\geq 0$, $\mu^N_t:=\frac1N\sum_{i=1}^N \delta_{X^{i,N}_t}$, which a.s. belongs to $\Pp(\R^2)$. If $\mu^N_0$ converges in probability, as $N\to \infty$, to $\rho_0$, then the family $(\mu^N)_{N\geq 2}$ is tight in $\Pp(C([0,\infty),\R^2))$ and any (possibly random) limit point 
$\mu$ of $(\mu^N)_{N\geq 2}$ a.s.  solves  \hyperref[defMP]{$\mathcal{(MP)}$} with initial law $\rho_0$.
\end{theorem}
The following statement is an immediate consequence of the above result.
\begin{cor}
With the assumptions and notations of Theorem~\ref{th:mainTH}, denoting by $(\mu_t)_{t\geq 0}$ a (possibly random) limit point of $((\mu^N_t)_{t\geq 0})_{N\geq 2}$ and defining $c_t=b_t^{c_0,\theta,\lambda}+ \int_0^t K^{\theta,\lambda}_{t-s} \ast \mu_s \,\dd s$, the couple $(\mu_t,c_t)_{t\geq 0}$ is a.s. a solution to \eqref{EDP_KS} with initial condition $(\rho_0,c_0)$ in the sense of Definition~\ref{def:soulutionPDE}. 
\end{cor}

Let us also mention that we have some 
weak  {\it regularity} estimates.

\begin{remark}
Adopt the assumptions and notations of Theorem~\ref{th:mainTH}.
There exists $\gamma \in (\frac32,2)$, depending on $\chi$, $\theta$ and $p$,
such that, with $\beta=2(\gamma-1)\in (1,2)$,
\begin{gather*} 
\sup_{N\geq 2} \E\Big[\int_0^t \frac{\dd s}{|X^{1,N}_s-X^{2,N}_s|^{\beta}}\Big]+
\E\Big[\int_0^t \int_{\R^2} \int_{\R^2}  \frac{1}{|x-y|^{\beta}} \mu_s(\dd y)\mu_s(\dd x) \dd s \Big]<\infty.
\end{gather*}
Moreover, we also  have \eqref{eq:Prop31-2}-\eqref{eq:Prop31-3prime}-\eqref{eq:Prop31-3} and \eqref{eq:ConditionMPplus2}.
\end{remark}

\paragraph{About the threshold.} As shown in Remark \ref{vn}, it for example holds true that 
(i) $\liminf_{\theta\to 0} \chi_{\theta,p}^* \geq 3.28$
for any $p>2$, (ii) $\chi_{1,p}^*\geq 1.39$ as soon as $p>3.3$, 
and (iii) $\liminf_{\theta\to \infty} \sqrt\theta\chi_{\theta,p}^* \geq 1.65$ as soon as $p>3.5$.
Moreover, we do not obtain better values of $\chi_{\theta,p}^*$ for larger values of $p$.

\medskip

In view of the global well-posedness/explosion results concerning \eqref{EDP_KS} mentioned above, we would of course prefer to have $\chi_{\theta,p}^*=8\pi$ for any $\theta>0$, at least
for large values of $p$ (or even with stronger regularity conditions on $c_0$). 
Our threshold is smaller, and this is due to the fact that it does not seem easy to make use, at the level of particles, of the {\it macroscopic} quantities exploited in \cite{CalCor2008}. Our proof is based on a tedious {\it a priori} estimate that relies on some moment computations and a suitable functional inequality.  
\medskip

According to \cite{Biler-Guerra-Karch,Corrias2014}, we could even hope for a very large threshold $\chi_{\theta,c_0}^*$ when $\theta$ is very large and $c_0$ is very small. We do not see how to modify our argument in this direction, because our method leads to a threshold independent of $c_0$ (at least if one assumes that $p=\infty$).
Observe that for $c_0$ given and non identically null, the threshold $\chi_{\theta,c_0}^*$ in \cite{Corrias2014} also tends to $0$ as $\theta\to \infty$.

\medskip

\paragraph{References.}
Particle approximations of singular P.D.E.s has been the subject of many papers. 
The closely related $2D$ Navier-Stokes and parabolic-elliptic Keller-Segel equations can be approximated by singularly interacting particle systems (singularity is of the order $1/r$, where $r$ is the pairwise particle distance) that are Markovian (not depending on the past of the particles). As already mentioned, we believe that the parabolic-parabolic equation also leads to a spatial singularity of order $1/r$. 
Something in this sense appears in the paragraph {\it Strategy} below. Moreover, observe that,
very roughly, if $X^{1,N}_t=X^{2,N}_s+R$, for some vector $R$, during the time interval $[t-1,t]$, then the corresponding interaction (in the drift of $X^{1,N}$) looks like, e.g. when $\lambda=0$,
$$
\int_{t-1}^t \nabla K^{\theta,\lambda}_{t-s} (X^{1,N}_t-X^{2,N}_s)
\dd s=-\frac{\theta R}{8\pi} \int_{t-1}^t \frac 1{(t-s)^2}
e^{-\frac{\theta}{4(t-s)}|R|^2} \dd s
= - \frac{R}{2\pi|R|^2}e^{-\frac{\theta |R|^2}{4}}
\sim - \frac{R}{2\pi |R|^2}
$$
as $|R| \to 0$. Of course, this is a caricatured situation.

\medskip

The convergence as $N\to\infty$ of an interacting particle system to the solution of the $2D$ Navier-Stokes equation has been established by Osada \cite{Osada1,Osada85long} (convergence along a subsequence for a large enough viscosity), Fournier-Hauray-Mischler \cite{FHM} (convergence of the whole sequence for an arbitrary positive viscosity) and finally by
Jabin-Wang \cite{JW} (quantitative convergence for an arbitrary positive viscosity).
The $2D$ parabolic-elliptic Keller-Segel equation has the same order of singularity as the $2D$ Navier-Stokes equation, but the interaction is {\it attractive} instead of being {\it rotating}, and solutions
do explode when the sensitivity parameter $\chi$ is greater than $8\pi$. Hence the situation is more delicate. The convergence of the associated particle system along a subsequence has been shown by Fournier-Jourdain~\cite{fournier-jourdain} when $\chi<2\pi$, and the quantitative convergence has been established by Bresch-Jabin-Wang~\cite{BJW} when $\chi<8\pi$, and the critical case $\chi=8\pi$ is treated in the work of Tardy~\cite{T}.
Let us also mention that Cattiaux-P\'ed\`eches \cite{CP} have proved the uniqueness in law
of the particle system in the subcritical case $\chi<8\pi$, which is far from obvious, and that a detailed study of the collisions
arising near the instant of explosion has been achieved,
in the supercritical case, by Fournier-Tardy in \cite{FournierTardy}. Finally,
Olivera-Richard-Toma\v sevi\'c\ \cite{ORT} are able to prove the quantitative convergence, in the supercritical case $\chi>8\pi$, until the explosion time, of a smoothed particle system, namely when the interaction in $1/r$  is replaced by an interaction in $1/(r+\e_N)$, with $\e_N=N^{-1/2+\eta}$ for some $\eta>0$.

\medskip

Concerning the parabolic-parabolic Keller-Segel equation, it seems that there are very few results about interacting particle systems. Jabir-Talay-Toma\v sevi\'c\ \cite{JTT} have considered the same problem as ours in dimension $1$. There the well-posedness and the propagation of chaos were proved using some suitable Girsanov transforms and without any constraints on the parameters of the model (as expected). Our situation here is somewhat more singular, as (i) in the $2D$ setting, solutions to the limit P.D.E. may explode in finite time for $\chi$ large and (ii) we expect that, as in the parabolic-elliptic case, particles do collide, even in the subcritical case, see \cite{fournier-jourdain}. Hence, it is not possible here to use Girsanov transforms as we do not expect the law of our system to be absolutely continuous w.r.t. the Wiener measure.
A more computational way to see this \textit{increase} of the singularity with the dimension is to note that the kernel $\nabla K$ satisfies, in dimension $d$, for all $p\geq 1$,
$$\|\nabla K\|_{L^p(\R^d)}=\frac{C_{p,d}}{t^{\frac{d}{2}(1-\frac{1}{p})+\frac{1}{2}}}.$$
In particular, $\nabla K$ belongs to $L^1((0,T);L^2 (\R^d))$ only for $d=1$ and this was crucial in \cite{JTT} for justifying the Girsanov transform.

\medskip
In any dimension $d\geq 1$, Stevens \cite{Stevens} studies a physically more satisfying particle system, with two populations (a population of cells and a population of chemo-attractant particles) in moderate interaction. The author proves the convergence in probability of the Kantorovich-Rubinstein distance between the empirical measure  of the  particle system and a solution to a generalized parabolic-parabolic Keller-Segel equation, including the supercritical case. The convergence is shown on the time interval where the limit P.D.E. has a solution belonging to $C^{1,3}_b([0,T]\times\R^d,\R)\cap C^0([0,T], L^2(\R^d))$. The moderate cutoff procedure, that we do not make precise here, decreases polynomially to $0$ as $N\to\infty$. 

\medskip
Let us mention that 
Chen-Wang-Yang \cite{CWY} also prove the convergence of a smoothed version of our  particle system, where 
$K^{\theta,\lambda}_{s}$ is replaced by $K^{\theta,\lambda}_{\e_N+s}$
with $\e_N=c(\log N)^{-\frac12}$ and, of course, they actually can work in any dimension. Such a result is obtained in two steps: first, a classical propagation of chaos with rate $e^{C/\e^2}N^{-1/2}$ using the Sznitman  coupling approach \cite{Sznitman} towards the smoothed limit equation, and then convergence of the smoothed equation with explicit rate.
In the same vein, in any dimension, Budhiraja-Fan \cite{Budhiraja} study a modified version of the doubly parabolic model, where the source term $\rho$ in the concentration equation is replaced by $\rho\ast g_1$. On the particle level,  $K^{\theta,\lambda}_{s}$ is replaced by $K^{\theta,\lambda}_{1+s}$ and the authors prove the trajectorial propagation of chaos, as in \cite{Sznitman}, and the uniform convergence of the associated Euler scheme.
The object of the present paper is rather to let first $\e\to 0$ and then $N\to \infty$.

\medskip
We also notice here that, in the two dimensional case, under a smallness condition on $\chi$, Toma\v sevi\'c~\cite{tomasevic.2} constructs the solution to a martingale problem related to~\eqref{eq:NLSDE}. The initial condition is supposed to be a probability density function in~\cite{tomasevic.2}, but the marginal laws of the solution to the martingale problem have some densities belonging to some mixed $L^q_t-L_x^p$ spaces. Even with more regularity on the initial condition, our method in the present paper does not allow us to find the martingale formulation of~\cite{tomasevic.2} as a limit of our particle system as $N\to \infty$. The main reason is that the only information we recover for the limit (along a subsequence) of the empirical measure of the particle system  is that its  marginal laws satisfy~\eqref{eq:ConditionMP}. We think it is very difficult to show, only using the latter, that some initial regularity propagates in time and that two martingale problems are equivalent. 

\medskip
As a conclusion, this work seems to be the first one deriving the $2D$ parabolic-parabolic Keller-Segel equation from a {\it non-smoothed} particle system. We are quite satisfied to show the existence of the particle system and its convergence along a subsequence even though we have a 
small threshold. However, at least for $\theta$ small or of order $1$ this threshold is non-ridiculously small. It should not come as a surprise that we only have convergence along a subsequence: Our notion of solution to~\eqref{EDP_KS} is so weak
that uniqueness seems very difficult to establish even with a smooth initial condition. 

\medskip 

Finally, observe that we have no assumption but exchangeability on the initial condition of the particle system: We can even start with all the particles at the same place.
In the same spirit,  we obtain a global existence result for \eqref{EDP_KS} derived from Theorem~\ref{th:mainTH} that is slightly different than the previous ones, as we only assume that $\rho_0 \in \Pp(\R^2)$
(which is already the case in \cite{Biler-Guerra-Karch}) and that $c_0 \in L^{2+\e}(\R^2)$, while all the previously cited papers work with $c_0\in H^1(\R^2)$.

\medskip

\paragraph{Strategy.}
The main point is to show that, when $\chi$ is small enough, something a bit stronger than~\eqref{PScond} holds true, uniformly in $N\geq 1$. 
To this end, we will introduce a $\e$-smoothed particle system, for which we will show 
that, considering e.g. 
particles $1$ and $2$ and 
setting $D^{1,2,N,\e}_s:=\int_0^s 
\nabla K^{\theta,\lambda}_{s-u}(X_s^{1,N,\e}-X_u^{2,N,\e}) \dd u$ 
\begin{equation}\label{ob2}
\sup_{N\geq 2} \E \Big[\int_0^t |D^{1,2,N,\e}_s|^{2(\gamma-1)}  \dd s\Big ]<\infty \qquad \text{for all $t>0$.}
\end{equation}
This estimate, which corresponds to \eqref{eq:Prop21-3} in Proposition \ref{prop:mc}, will be shown uniformly in $\e$.
This will allow us to pass to the limit as $\e\to 0$ in order to prove the existence of the non-smoothed particle system. This system satisfies, when setting $D^{1,2,N}_s:=\int_0^s 
\nabla K^{\theta,\lambda}_{s-u}(X_s^{1,N}-X_u^{2,N}) \dd u$,
\begin{equation}\label{ob}
\sup_{N\geq 2} \E \Big[\int_0^t |D^{1,2,N}_s|^{2(\gamma-1)}  \dd s\Big ]<\infty \qquad \text{for all $t>0$.}
\end{equation}
see \eqref{eq:Prop31-3} in the proof, and then to pass to the limit
as $N\to \infty$.

\medskip

Let us show why \eqref{ob} {\it a priori} holds. The difficulty when proving \eqref{ob} is to show that particles are not too close to each other. Indeed, $|D^{1,2,N}_s|$  may explode only if $X^{1,N}_s=X^{2,N}_s$, because $(u,x)=(s,0)$ is the only problematic point of $\nabla K^{\theta,\lambda}_{s-u}(x)$. Otherwise, the integral is well defined. We show in Proposition~\ref{lemma:LL} that
\begin{equation}\label{mark}
\E\Big[\int_0^t |D^{1,2,N}_s|^{2(\gamma-1)}  \dd s\Big] \leq C \E\Big[\int_0^t |X^{1,N}_s-X^{2,N}_s|^{-2(\gamma-1)}  \dd s\Big].
\end{equation}
The proof of this estimate is difficult to summarize and relies on a slightly special application of the It\^o formula shown in Lemma~\ref{lemma:grosIto}, used with a well-chosen function, on a bound of $|\nabla K^{\theta,\lambda}|$ obtained in Remark \ref{ttt}, and on a key functional inequality proved in Lemma~\ref{lemma:FI}.

\medskip
Observe that \eqref{mark} reveals, in some loose sense, that the drift term $D^{1,2,N}_s$ is controlled
by $|X^{1,N}_s-X^{2,N}_s|^{-1}$, which is precisely the singularity of the drift in the parabolic-elliptic case. This is what we call a \textit{Markovianization}:
we bound the {\it path dependent} interaction by a {\it current time dependent} one.

\medskip

Once this Markovianization is performed in Proposition \ref{lemma:LL}, we may 
conclude the proof of Proposition \ref{prop:mc} (including \eqref{ob}) by applying the strategy of \cite{fournier-jourdain}, that has been refined in \cite{FT,T}: Applying the It\^o formula,
and using exchangeability, we find (assuming that $c_0=0$ for simplicity)
\begin{align*}
\E[|X^{1,N}_t-X^{2,N}_t|^{4-2\gamma}]=&
\E[|X^{1,N}_0-X^{2,N}_0|^{4-2\gamma}]
+(4-2\gamma)^2  \int_0^t\E[|X^{1,N}_s-X^{2,N}_s|^{2-2\gamma}]\dd s\\
&+ \frac{4-2\gamma}{N-1}\chi \sum_{j=2}^N\int_0^t\E[|X^{1,N}_s-X^{2,N}_s|^{2-2\gamma} (X^{1,N}_s-X^{2,N}_s)\cdot D^{1,j,N}_s]\dd s.
\end{align*}
Using exchangeability again, \eqref{mark} and the H\"older inequality, one may control 
the last term by $C\chi \E[\int_0^t |X^{1,N}_s-X^{2,N}_s|^{-2(\gamma-1)}\dd s]$.
All this shows that 
$$
\E[|X^{1,N}_t-X^{2,N}_t|^{4-2\gamma}] \geq \E[|X^{1,N}_0-X^{2,N}_0|^{4-2\gamma}] 
+ ((4-2\gamma)^2- C \chi) \E\Big[\int_0^t |X^{1,N}_s-X^{2,N}_s|^{-2(\gamma-1)}\dd s\Big].
$$
Since now particles are subjected to attraction, there is no reason why they should be far from $0$. Since $4-2\gamma>0$,  we expect that
$\E[|X^{1,N}_t-X^{2,N}_t|^{4-2\gamma}]$ should be easily controlled, uniformly in $N\geq2$, by some constant $A_t$ 
(actually, we use as in \cite{FT} a slightly
more clever function than $|\cdot |^{4-2\gamma}$ and this last argument is useless), and we end with 
$$
((4-2\gamma)^2-C\chi) \E\Big[\int_0^t |X^{1,N}_s-X^{2,N}_s|^{-2(\gamma-1)}\dd s\Big] \leq A_t.
$$
Combined with \eqref{mark}, this gives us \eqref{ob} provided that $\chi<(4-2\gamma)/C$.

\medskip

\paragraph{Key functional inequality.} 
The following functional inequality, that we will prove in Appendix~\ref{sec:app} plays a central role in our main computation.

\begin{lemma}
\label{lemma:FI}
Let $b>a>0$ and $t>0$. For any measurable function $f:[0, t] \to \R_+  $, we have 
$$
\int_0^t \frac1{(s+ f(s))^{1+a}} \dd s\leq \kappa(a,b) \Big(\int_0^t \frac{1}{(s+ f(s))^{1+b}} \dd s\Big)^{\frac{a}{b}},\quad \hbox{where} \quad
\kappa(a,b)=\frac{a+1}{a}\Big[\frac{b}{b+1}\Big]^{\frac a b}.
$$
\end{lemma}
The constant $\kappa(a,b)$ is optimal (for any value of $t>0$), as one can show by choosing
$f(s)=(\e-s)_+$ and by letting $\e\to0$. 

\paragraph{Plan of the paper.} In Section~\ref{sec:main}, we start from a regularized particle system and we present the main computation of the paper. 
Once this is done, Sections \ref{sec:tight}-\ref{sec:ex}-\ref{sec:conc} contain respectively a tightness result (Lemma~\ref{tight}), the existence result for the particle system (Proposition~\ref{prop:ex}) and that any limit point satisfies  \hyperref[defMP]{$\mathcal{(MP)}$} (Theorem~\ref{conv}). In Section~\ref{sec:condition}   we discuss  our smallness condition on the chemotactic sensitivity $\chi$ and we make it more explicit. Finally, in Appendix~\ref{sec:app} we prove Lemma~\ref{lemma:FI}. 

\section{Main computation}
\label{sec:main}
For $c_0\in L^p(\R^2)$ with $p>2$, for $\e\in (0,1]$ and $N\geq 2$, we introduce a smoothed version of the interaction kernel
\begin{equation}\label{Heps}
H_t^{\theta,\lambda, \varepsilon} (x):= \frac{t^2}{(t+\varepsilon)^2} \nabla K_t^{\theta,\lambda}(x)= -\frac{\theta}{8\pi (t+\varepsilon)^2} e^{-\frac{\lambda}{\theta}t } e^{- \frac{\theta}{4t}|x|^2} x,
\end{equation}
as well as the smoothed version of the Keller-Segel particle system: For all $i=1,\dots,N$,
\begin{align}
X_t^{i,N,\varepsilon}\!= &X_0^{i,N} \!+\! \sqrt{2} W_t^i \!+\! \chi \int_0^t \!\!\nabla b_{s+\e}^{c_0,\theta,\lambda}(X^{i,N,\e}_s)\dd s\!+\! \frac{\chi}{N-1} \sum_{j\neq i}\int_0^t \!\!\int_0^s \!\!H^{\theta,\lambda, \varepsilon}_{s-u} ( X^{i,N,\varepsilon}_s\!-\!X^{j,N,\varepsilon}_u) \dd u \dd s.
\label{eq:PSS}
\end{align} 
This system has a pathwise unique solution, as $\nabla b^{c_0,\theta,\lambda}_{t+\e}$ and $H_t^{\theta,\lambda, \varepsilon}$ are globally Lipschitz continuous, uniformly in $t\geq 0$. If the initial condition $(X^{i,N}_0)_{i=1,\dots,N}$ is exchangeable, then the family $((X^{i,N,\e}_t)_{t\geq 0})_{i=1,\dots,N}$ is also exchangeable by uniqueness in law.

\medskip

The constants $\kappa(a,b)$, for $b>a>0$, are defined in Lemma~\ref{lemma:FI} and for $\alpha,\beta,\theta>0$ and $\gamma>3/2$, we introduce
\begin{gather}
    \label{eq:C0}
    C_0(\beta):=\sup_{u\ge 0} \sqrt u (1+\beta u) ^{3/2} e^{-u}, \qquad
    C_1(\alpha,\gamma):=(\gamma-1)(1-4\alpha(\gamma-1)),\\
    \label{eq:C2}
    C_2(\theta,\alpha,\gamma):= \frac{\sqrt{\alpha\theta}(\gamma-1)}{2\pi}C_0\Big(\frac{4\alpha}{\theta}\Big) \kappa\Big(\frac{1}{2}, \gamma-1\Big)\kappa\Big(\gamma-\frac32,\gamma-1\Big).
\end{gather}
The goal of this section is to prove the following estimates, from which our main theorem will be more or less classically deduced, see e.g. Osada \cite{Osada85long}
and then \cite{FHM,fournier-jourdain}.
\begin{proposition}\label{prop:mc}
Assume that for each $N\geq 2$, the family $(X^{i,N}_0)_{i=1,\dots,N}$ is exchangeable
and that $c_0 \in L^p(\R^2)$ for some $p>2$.
Let $\gamma\in(\frac{3}{2}, \frac{2p+2}{p+2})$ and $\alpha>0$ such that $C_1(\alpha,\gamma)>0$. Assume that $\chi>0$ and $\theta>0$ are such that
$$C_1(\alpha,\gamma)>\chi C_2(\theta,\alpha,\gamma)
\qquad \hbox{and} \qquad (4-2\gamma)-\chi \frac{\sqrt{\theta} C_0\big(\frac{4\alpha}{\theta}\big) \kappa\big(\frac{1}{2}, \gamma-1\big)}{4\pi\sqrt{\alpha}[C_1(\alpha,\gamma) -\chi C_2(\theta,\alpha,\gamma)]^{\frac 1{2(\gamma-1)}}}>0.$$
Then for all $t>0$,
\begin{gather}
\label{eq:Prop21-1}
\sup_{\varepsilon\in(0,1], N\geq 2} \E \Big[\int_0^t  \frac{1}{|X_s^{1,N,\varepsilon}-X_s^{2,N,\varepsilon}|^{2(\gamma-1)}} \dd s\Big ]<\infty,\\
\label{eq:Prop21-2}
\sup_{\varepsilon\in(0,1], N\geq 2} \E \Big[\int_0^t  \int_0^s \frac{1}{(s-u+ |X_s^{1,N,\varepsilon}-X_u^{2,N,\varepsilon}|^{2})^\gamma} \dd u \dd s \Big ]<\infty,\\
\label{eq:Prop21-3prime}
\sup_{\varepsilon\in(0,1], N\geq 2} \E \Big[\int_0^t \int_0^s |\nabla K^{\theta,\lambda}_{s-u}(X_s^{1,N,\varepsilon}-X_u^{2,N,\varepsilon})|^{\frac{2\gamma} 3} \dd u   \dd s\Big ]<\infty,\\
\label{eq:Prop21-3}
\sup_{\varepsilon\in(0,1], N\geq 2} \E \Big[\int_0^t \Big( \int_0^s |\nabla K^{\theta,\lambda}_{s-u}(X_s^{1,N,\varepsilon}-X_u^{2,N,\varepsilon})| \dd u \Big)^{2(\gamma-1)}  \dd s\Big ]<\infty.
\end{gather}
\end{proposition}
It is important to notice that the exponents $2(\gamma-1)$ and $\frac{2\gamma}{3}$ are both greater than $1$.
For some comments about the interest of these estimates and the strategy to prove them, we refer to the paragraph {\it Strategy} in Section~\ref{inin}. 
Let us first make the following observation.

\begin{remark}\label{ttt}
For any $\e\in (0,1]$, any $\alpha>0$, any $t>0$, any $x \in \R^2$, we have
$$
|\nabla K_t^{\theta,\lambda} (x)| \leq 
\frac{\sqrt \theta C_0\big(\frac{4\alpha}{\theta}\big)}{4\pi(t+\alpha |x|^2)^{\frac{3}{2}}}
\qquad \hbox{and} \qquad
|H_t^{\theta,\lambda, \varepsilon} (x)| \leq 
\frac{\sqrt \theta C_0\big(\frac{4\alpha}{\theta}\big)}{4\pi(t+\e+\alpha |x|^2)^{\frac{3}{2}}}.
$$
\end{remark}
\begin{proof}
Let us for example study the case of $H^{\theta,\lambda,\e}_t(x)$. We write
$$|H_t^{\theta,\lambda, \varepsilon} (x)|= \frac{\theta}{8\pi (t+\varepsilon)^2} e^{-\frac{\lambda}{\theta}t } e^{- \frac{\theta}{4t}|x|^2} |x|\leq
\frac{1}{(t+\e)^{3/2}} \frac{\theta}{8\pi}\frac{|x|} {(t+\e)^{1/2}}  e^{- \frac{\theta}{4(t+\e)}|x|^2}.
$$
Setting $u=\frac{\theta}{4(t+\e)}|x|^2$, this rewrites 
$$
|H_t^{\theta,\lambda, \varepsilon} (x)|\leq
\frac{1}{(t+\e)^{3/2}} \frac{\sqrt \theta}{4\pi} \sqrt u e^{-u}
\leq \frac{1}{(t+\e)^{3/2}} \frac{\sqrt \theta}{4\pi} \frac{C_0\big(\frac{4\alpha}\theta\big)}{(1+\frac{4\alpha}\theta u)^{3/2}}
$$
by definition \eqref{eq:C0} of $C_0$. The result is proved, since
$(t+\e)(1+\frac{4\alpha}\theta u)=t+\e+\alpha |x|^2$.
\end{proof}

During the whole section we drop the superscript $N,\varepsilon$, i.e. we write $X^{i}_t=X^{i,N,\e}_t$, but we keep in mind that all the estimates have to be uniform in these parameters.  In addition, we define
\begin{gather*}
R_{t,s}^{i,j}:= X^{i}_t-X^{j}_s \qquad \hbox{and} \qquad
D^{i,j}_t:=\int_0^t H^{\theta,\lambda,\varepsilon}_{t-s} (R_{t,s}^{i,j})  \dd s.
\end{gather*}
We start with the following It\^o formula.
\begin{lemma}
\label{lemma:grosIto}
Let $F: \R_+ \times \R^2 \to \R$ be of class $C^{1,2}_b(\R_+ \times \R^2)$. For all $t>0$,
\begin{align*}
    \E \Big[ \int_0^t F(t-s, R_{t,s}^{1,2}) \dd s \Big ] =& \E \Big[ \int_0^t F(0, R_{s,s}^{1,2}) \dd s \Big ] + \E \Big[\int_0^t \int_0^u (\partial_t F + \Delta F) (u-s, R_{u,s}^{1,2})  \dd s \dd u \Big]\\
    &+\chi \E\Big[\int_0^t \Big(\int_0^u \nabla F (u-s, R_{u,s}^{1,2}) \dd s \Big)   \cdot \nabla b^{c_0,\theta,\lambda}_{u+\e}(X^{1}_u) \dd u\Big]\\
    &+ \frac{\chi}{N-1} \sum_{j=2}^N \E \Big[\int_0^t \Big(\int_0^u \nabla F (u-s, R_{u,s}^{1,2})  \dd s \Big)  \cdot D^{1,j}_u \dd u \Big].
\end{align*}
\end{lemma}
\begin{proof}
Recalling \eqref{eq:PSS}, we have, for $t>s>0$,
$$
R^{1,2}_{t,s}=R^{1,2}_{s,s}+X^1_t-X^1_s=R^{1,2}_{s,s}+ \sqrt 2 (W^{1}_t-W^1_s)+\chi \int_s^t \!\!\nabla b_{u+\e}^{c_0,\theta,\lambda}(X^{1}_u)\dd u\!+\! \frac{\chi}{N-1} \sum_{j=2}^N\int_s^t D^{1,j}_{u} \dd u.
$$
Applying the It\^o formula on $[s,t]$, we find 
\begin{align*}
   \E [ F(t-&s, R_{t,s}^{1,2}) ] =  \E [ F(0, R_{s,s}^{1,2})]+ \E \Big[\int_s^t  (\partial_t F + \Delta F) (u-s, R_{u,s}^{1,2})   \dd u \Big]\\
    &+\chi \E\Big[\int_s^t  \nabla F (u-s, R_{u,s}^{1,2})    \cdot \nabla b^{c_0,\theta,\lambda}_{u+\e}(X^{1}_u) \dd u\Big]+ \frac{\chi}{N-1} \sum_{j=2}^N \E \Big[\int_s^t  \nabla F (u-s, R_{u,s}^{1,2})    \cdot D^{1,j}_u \dd u\Big].
\end{align*}
Integrating the formula in $s$ on $[0,t]$ and applying the Fubini theorem completes the proof.
\end{proof}

The next result shows that, in some loose sense, we can reduce to the parabolic-elliptic case. Gathering the estimates below, we see that
$\E[\int_0^t |D^{1,2}_u|^{2(\gamma-1)}\dd u] \leq C \E[\int_0^t|R^{1,2}_{u,u}|^{2(1-\gamma)}]+C$.
Hence, roughly, we control the drift $|D^{1,2}_u|$ by  $|R^{1,2}_{u,u}|^{-1}$, which does not depend on the past of the particles and has the homogeneity of the drift of the parabolic-elliptic particle system.

\begin{proposition}\label{lemma:LL} 
Assume that $(X^{i,N}_0)_{i=1,\dots,N}$ is exchangeable and that $c_0 \in L^p(\R^2)$ for some $p>2$.
Consider $\chi>0$, $\theta>0$, $\gamma  \in (\frac 32,\frac{2p+2}{p+2})$
and $\alpha>0$ such that
$C_1(\alpha,\gamma)>0$ and $\chi C_2(\theta,\alpha,\gamma)<C_1(\alpha,\gamma)$.
Then for all $\varepsilon\in(0,1]$, all $N\geq 1$, all $t>0$ and all $j=2,\dots,N$,
\begin{equation}
\label{eq:borneDrift}
|D^{1,j}_t |\leq \frac{\sqrt \theta C_0\big(\frac{4\alpha}\theta\big)\kappa\big(\frac{1}{2}, \gamma-1\big)}{4\pi}   (S^{1,j}_t)^{\frac{1}{2(\gamma-1)}}
\qquad \hbox{where}\qquad S^{1,j}_t:= \int_0^t \frac{1}{(t-s + \alpha |R^{1,j}_{t,s}|^2)^{\gamma}} \dd s,
\end{equation}
and for any $\eta>0$, there is a constant $A_\eta=A_\eta(c_0,p,\chi,\theta,\gamma,\alpha)$ such that 
\begin{equation}\label{eq:borneSt}
\E\Big[ \int_0^t S_u^{1,j} \dd u \Big]\leq \frac{(1+\eta)\alpha^{1-\gamma}}{C_1(\alpha,\gamma) -\chi C_2(\theta,\alpha,\gamma)} \E\Big[ \int_0^t \frac{1}{|R_{u,u}^{1,j}|^{2(\gamma-1)}}\dd u \Big] + A_\eta t^{r_p}
\end{equation}
for all $t>0$ and all $j=1,\dots,N$,
where $r_p=1-(\gamma-1)(1+\frac2p)>0$ (because $\gamma<\frac{2p+2}{p+2}$).
\end{proposition}

Observe that for any $\chi>0$, any $\theta>0$ and any $\gamma\in (\frac32,\frac{2p+2}{p+2})$, the two conditions
$C_1(\alpha,\gamma)>0$ and $\chi C_2(\theta,\alpha,\gamma)<C_1(\alpha,\gamma)$ are satisfied
for $\alpha>0$ small enough, because $\lim_{\alpha\to 0} C_2(\theta,\alpha,\gamma)=0$ and  $\lim_{\alpha\to 0} C_1(\alpha,\gamma)=\gamma-1$. The real restrictions will come later.

\begin{proof}
It of course suffices to treat the case where $j=2$.

\medskip

{\it Step 1.} We first prove \eqref{eq:borneDrift}. 
By Remark \ref{ttt}, it holds that
$$|D^{1,2}_t|\leq \frac{\sqrt \theta C_0\big(\frac{4\alpha}\theta\big)}{4\pi} \int_0^t \frac{1}{(t-s+\e+\alpha |R^{1,2}_{t,s}|^2)^{\frac{3}{2}}} \dd s
= \frac{\sqrt \theta C_0\big(\frac{4\alpha}\theta\big)}{4\pi} \int_0^t \frac{1}{(s+\e+\alpha |R^{1,2}_{t,t-s}|^2)^{\frac{3}{2}}} \dd s
.$$
Applying Lemma~\ref{lemma:FI} with $a=\frac{1}{2}$,  $b=\gamma-1$ and $f(s)= \alpha |R^{1,2}_{t,t-s}|^2$, it comes
\begin{align}\label{newS}
|D^{1,2}_t|\leq & \frac{\sqrt\theta C_0\big(\frac{4\alpha}\theta\big)\kappa\big(\frac{1}{2}, \gamma-1\big)}{4\pi}  \Big( \int_0^t \frac{1}{(s+\e+\alpha |R^{1,2}_{t,t-s}|^2)^{\gamma}} \dd s\Big)^{\frac{1}{2(\gamma-1)}}\notag\\
=& \frac{\sqrt \theta C_0\big(\frac{4\alpha}\theta\big)\kappa\big(\frac{1}{2}, \gamma-1\big)}{4\pi}   (\bar S^{1,2,\e}_t)^{\frac{1}{2(\gamma-1)}},
\end{align}
where for any $\delta>0$, we have set
\begin{equation}\label{barS}
\bar S^{1,j,\delta}_t:= \int_0^t \frac{1}{(t-s +\delta+ \alpha |R^{1,j}_{t,s}|^2)^{\gamma}} \dd s \leq S^{1,j}_t.
\end{equation}

\textit{Step 2.} Let $F(t, x)= - (t+\alpha|x|^2)^{1-\gamma}$. We have 
$$\nabla F(t,x)= 2\alpha (\gamma-1) (t+\alpha|x|^2)^{-\gamma} x$$
and 
\begin{align*}
(\partial_t F + \Delta F)(t,x) =& (\gamma-1)(t+\alpha|x|^2)^{-\gamma-1} [(1+4\alpha)(t + \alpha |x|^2)- 4\alpha^2\gamma |x|^2]\\
\geq& (\gamma-1)(1+4\alpha-4\alpha\gamma)(t+\alpha|x|^2)^{-\gamma} \\
=&C_1(\alpha,\gamma)(t+\alpha|x|^2)^{-\gamma}.
\end{align*}

\textit{Step 3.} We now prove \eqref{eq:borneSt}.  We apply Lemma~\ref{lemma:grosIto} with the function $F$ introduced in Step~2, or rather with the smooth function $F(\delta+t,x)$, for some $\delta>0$ that we will tend to $0$. At first reed, one can take $\delta=0$: The computations are then slightly informal since $F$ is not smooth, but this makes disappear some terms that are  actually not very important. 
We find
\begin{equation}\label{eq:abeg}
I^{\delta,1}_t= I^{\delta,2}_t+I^{\delta,3}_t +I^{\delta,4}_t+\frac{1}{N-1} \sum_{j=2}^N A^{\delta,1,j}_t,
\end{equation}
where 
\begin{align}
I^{\delta,1}_t:=&\E \Big[\int_0^t F(\delta+t-s, R^{1,2}_{t,s}) \dd s\Big]\leq 0, \label{eq:a}\\
I^{\delta,2}_t :=&\E \Big[\int_0^t F(\delta, R^{1,2}_{s,s}) \dd s\Big]= -\E \Big[\int_0^t (\delta + \alpha |R^{1,2}_{s,s}|^2)^{1-\gamma} \dd s\Big],\label{eq:b}\\
I^{\delta,3}_t:=&  \E \Big[\int_0^t \int_0^u (\partial_t F + \Delta F) (\delta+u-s, R_{u,s}^{1,2})  \dd s \dd u \Big], \notag\\ 
I^{\delta,4}_t:=&\chi \E\Big[\int_0^t \Big(\int_0^u \nabla F (\delta+u-s, R_{u,s}^{1,2}) \dd s \Big)   \cdot \nabla b^{c_0,\theta,\lambda}_{u+\e}(X^{1}_u) \dd u\Big], \notag\\
A^{\delta,1,j}_t:=&  \chi \E \Big[\int_0^t \Big(\int_0^u  \nabla F(\delta+u-s, R_{u,s}^{1,2})     \dd s\Big)\cdot D^{1,j}_u \dd u\Big]. \notag
\end{align}
Using Step 2, we find
\begin{align}    
I^{\delta,3}_t\geq& C_1(\alpha,\gamma)\E\Big[\int_0^t\int_0^u(\delta+u-s+\alpha|R^{1,2}_{u,s}|^2)^{-\gamma} \dd s \dd u \Big]
= C_1(\alpha,\gamma)\E\Big[\int_0^t \bar S^{1,2,\delta}_u \dd u \Big], \label{eq:e}
\end{align}
recall \eqref{barS}. Next, we write
\begin{align}\label{eq:calculGjt}
A^{\delta,1,j}_t \geq - \chi \E \Big[\int_0^t  |D^{1,j}_u| T_u \dd u\Big],
\end{align}
with $T_u=\int_0^u  |\nabla F(\delta+u-s, R_{u,s}^{1,2})| \dd s= \int_0^u |\nabla F(\delta+s, R_{u,u-s}^{1,2})| \dd s $. By Step 2 again,
\begin{align*}
T_u=&2\alpha(\gamma-1)\int_0^u \frac{|R^{1,2}_{u,u-s}| \dd s}{(s+\delta +\alpha| R^{1,2}_{u,u-s}|^2)^{\gamma}}\leq 2\sqrt{\alpha}(\gamma-1)\int_0^u \frac{\dd s}{(s+\delta+\alpha| R^{1,2}_{u,u-s}|^2)^{\gamma-1/2}}.
\end{align*}
Apply Lemma~\ref{lemma:FI} with $a=\gamma-3/2 $, $b=\gamma-1$   and $f(s)= \delta+\alpha |R^{1,j}_{u,u-s}|^2$. It comes
\begin{align}
T_u\leq & 2\sqrt{\alpha}(\gamma-1)\kappa\Big(\gamma-\frac32,\gamma-1\Big)\Big(\int_0^u \frac{\dd s}{(s+\delta+\alpha| R^{1,2}_{u,u-s}|^2)^{\gamma}}\Big)^ {\frac{\gamma-3/2}{\gamma-1}} \notag \\
=&2\sqrt{\alpha}(\gamma-1)\kappa\Big(\gamma-\frac32,\gamma-1\Big)\Big(\bar S^{1,2,\delta}_u\Big)^ {\frac{\gamma-3/2}{\gamma-1}}. \label{estTu}
\end{align}
This last inequality, plugged together with \eqref{newS} in \eqref{eq:calculGjt}, gives us 
\begin{align*}
A_t^{\delta,1,j}\geq& -2\chi\sqrt{\alpha}(\gamma-1)\frac{\sqrt \theta C_0\big(\frac{4\alpha}\theta\big)\kappa\big(\frac{1}{2}, \gamma-1\big)\kappa\big(\gamma-\frac32,\gamma-1\big)}{4\pi}   \E \Big[\int_0^t   \Big(\bar S^{1,2,\delta}_u\Big)^ {\frac{\gamma-3/2}{\gamma-1}} (\bar S^{1,j,\e}_u)^{\frac{1}{2(\gamma-1)}}\dd u \Big]\\
=& -\chi C_2(\theta,\alpha,\gamma) \E \Big[\int_0^t   \Big(\bar S^{1,2,\delta}_u\Big)^ {\frac{\gamma-3/2}{\gamma-1}} (\bar S^{1,j,\e}_u)^{\frac{1}{2(\gamma-1)}}\dd u \Big]\\
\geq & -\chi C_2(\theta,\alpha,\gamma)\E \Big[\int_0^t   \Big(\bar S^{1,2,\delta}_u\Big)^ {\frac{\gamma-3/2}{\gamma-1}}
(\bar S^{1,j,\delta}_u)^{\frac{1}{2(\gamma-1)}}\dd u \Big]
\end{align*}
if $\delta \in (0,\e]$.
By the H\"older inequality (both for $\E$ and $\int_0^t$) with $p=\frac{\gamma-1}{\gamma-3/2}$ and $p'=2(\gamma-1)$,
\begin{align}
A_t^{\delta,1,j}\geq& -\chi C_2(\theta,\alpha,\gamma)
\Big(\E \Big[\int_0^t   \bar S^{1,2,\delta}_u\dd u\Big]\Big)^{\frac{\gamma-3/2}{\gamma-1}} \Big(\E\Big[\int_0^t \bar S^{1,j,\delta}_u\dd u \Big]\Big)^{\frac{1}{2(\gamma-1)}}\notag\\
=&-\chi C_2(\theta,\alpha,\gamma) \E \Big[\int_0^t   \bar S^{1,2,\delta}_u\dd u\Big] \label{eq:Gjt}
\end{align}
by exchangeability. Finally, recalling \eqref{estib} and \eqref{estTu}, for some constant
$A=A(c_0,p,\chi,\theta,\gamma,\alpha)$ that may change from line to line,
\begin{align*}
I^{\delta,4}_t \geq -A \E\Big[\int_0^t T_u \frac{\dd u}{u^{\frac12+\frac1p}}\Big]
\geq -A \E\Big[\int_0^t (\bar S^{1,2,\delta}_u)^{\frac{\gamma-3/2}{\gamma-1}} \frac{\dd u}{u^{\frac12+\frac1p}}\Big].
\end{align*}
Using the Young inequality with $p=\frac{\gamma-1}{\gamma-3/2}$ and $p'=2(\gamma-1)$,
we find that for any $\zeta>0$, there is a constant $A_\zeta=A_\zeta(c_0,p,\chi,\theta,\gamma,\alpha)$
such that
\begin{equation}\label{eq:4delta}
I^{\delta,4}_t \geq -\zeta\E\Big[\int_0^t \bar S^{1,2,\delta}_u\dd u\Big]
-A_\zeta \int_0^t \frac{\dd u}{u^{(\gamma-1)(1+\frac2p)}} \dd u \geq -\zeta\E\Big[\int_0^t \bar S^{1,2,\delta}_u\dd u\Big]
-A_\zeta t^{r_p},
\end{equation}
recall that $r_p=1-(\gamma-1)(1+\frac2p)>0$.

\medskip

Plugging \eqref{eq:a}, \eqref{eq:b}, \eqref{eq:e}, \eqref{eq:Gjt}, \eqref{eq:4delta} into \eqref{eq:abeg} we obtain, for any $\delta \in (0,\e]$, any $\zeta>0$, 
\begin{align*}
(C_1(\alpha,\gamma)-\chi C_2(\theta,\alpha,\gamma)-\zeta) \E \Big[\int_0^t   \bar S^{1,2,\delta}_u\dd u\Big] \leq &\E \Big[\int_0^t (\delta + \alpha |R^{1,2}_{s,s}|^2)^{1-\gamma} \dd s\Big]+A_\zeta t^{r_p}\\ 
\leq& \alpha^{1-\gamma}  \E \Big[\int_0^t |R^{1,2}_{s,s}|^{2(1-\gamma)} \dd s\Big]+A_\zeta t^{r_p} .
\end{align*}
Letting $\delta\to 0$, we find that if $\zeta\in (0,C_1(\alpha,\gamma)-\chi C_2(\theta,\alpha,\gamma))$,
$$
\E \Big[\int_0^t S^{1,2}_u\dd u\Big]
\leq \frac{\alpha^{1-\gamma}}{C_1(\alpha,\gamma)-\chi C_2(\theta,\alpha,\gamma)-\zeta}  \E \Big[\int_0^t |R^{1,2}_{s,s}|^{2(1-\gamma)} \dd s\Big]
+A_\zeta t^{r_p}.
$$
The conclusion immediately follows.
\end{proof}

We are now ready to conclude this section. We use here some ideas of \cite{FT}, which is natural since  in Proposition~\ref{lemma:LL}, we loosely showed that the singularity here is of the same order as in the parabolic-elliptic case. In particular, 
we borrow the functions $\phi$ and $\psi$ below.

\begin{proof}[Proof of Proposition~\ref{prop:mc}]
We set $\nu=4-2\gamma \in (0,1)$ and divide the proof in several steps.

\medskip

{\it Step 1.} We introduce the function $\phi(r):=(1+ r^\frac{\nu}{2})^{-1}r^\frac{\nu}{2}$
on $\R_+$ and set $\psi(x,y)=\phi(|x-y|^2)$ for $x,y\in \R^2$. As in \cite[Proof of Proposition~5]{FT}, it holds that
$$\nabla_x \psi (x,y)= \nu \frac{|x-y|^{\nu-2}}{(1+|x-y|^\nu)^2}(x-y)\quad \hbox{and}\quad \Delta_x \psi (x,y)= \nu^2 \frac{|x-y|^{\nu-2}}{(1+|x-y|^\nu)^2}\Big(1-2 \frac{|x-y|^{\nu}}{1+|x-y|^\nu}\Big).$$
For any $\eta >0$, there exists a constant $L_{\eta}=L_\eta(\gamma)>0$ (recall that $\nu=4-2\gamma$)
such that
$$
\Delta_x \psi (x,y) \geq (\nu^2-\eta) |x-y|^{\nu-2}-L_\eta
$$ 
for all $x,y\in\R^2$. To check this claim, it suffices to prove that the function
$$f_\eta(r)= \nu^2 \frac{r^{\nu-2}}{(1+r^\nu)^2}\Big(1-2\frac{r^{\nu}}{1+r^\nu}\Big)-(\nu^2-\eta)r^{\nu-2}
$$
is bounded from below (possibly by a negative constant) on $(0,\infty)$. This follows from the facts that $f_\eta$ is continuous on $(0,\infty)$ and that $\lim_{r\to 0}f_\eta(r)=+\infty$ and $\lim_{r\to +\infty}f_\eta(r)=0$.

\medskip

{\it Step 2.}
Applying the It\^o formula (as in the previous proof, one should first consider a smooth approximation of $\psi$, but we will not repeat this here), we obtain
\begin{align*}
\E[\psi (X^1_t&,X^2_t)]=\E[\psi (X^1_0,X^2_0)]  + \E\Big[\int_0^t [\Delta_x \psi (X^1_s,X^2_s)+\Delta_y\psi (X^1_s,X^2_s)] \dd s\Big] \\
&+ \E\Big[\int_0^t [\nabla_x \psi (X^1_s,X^2_s)\cdot \nabla b_{s+\e}^{c_0,\theta,\lambda}(X^1_s)+\nabla_y \psi (X^1_s,X^2_s)\cdot \nabla b_{s+\e}^{c_0,\theta,\lambda}(X^2_s)] \dd s\Big] \\
& + \frac{\chi}{N-1} \sum_{j\neq 1}\E\Big[\int_0^t  \nabla_x \psi (X^1_s,X^2_s) \cdot D^{1,j}_s\dd s  \Big]+ \frac{\chi}{N-1} \sum_{j\neq 2}\E\Big[\int_0^t \nabla_y\psi (X^1_s,X^2_s) \cdot D^{2,j}_s\dd s\Big].
\end{align*}
By symmetry of $\psi$ and exchangeability of the particle system, we have
\begin{equation}
\label{eq:plug}
J^1_t=J^1_0+ 2J^2_t + 2J^3_t+ \frac{2}{N-1} \sum_{j=2}^N B^{1,j}_t,
\end{equation}
where 
\begin{gather}
J^1_t:= \E[\psi (X^1_t,X^2_t)]\leq 1, \qquad\qquad
J^1_0:= \E[\psi (X^1_0,X^2_0)]\geq 0, \label{eq:aaa}
\end{gather}
and where
\begin{align*}    
J^2_t:=& \E\Big[\int_0^t \Delta_x \psi (X^1_s,X^2_s) \dd s\Big],\\
J^3_t:=&\chi\E\Big[\int_0^t \nabla_x \psi (X^1_s,X^2_s)\cdot \nabla b_{s+\e}^{c_0,\theta,\lambda}(X^1_s) \dd s \Big],\\
B^{1,j}_t:=&\chi\E\Big[\int_0^t  \nabla_x \psi (X^1_s,X^2_s)\cdot  D^{1,j}_s\dd s \Big].
\end{align*}
By Step 1, it holds that for any $\eta>0$,
\begin{equation}
    \label{eq:bbb}
    J^2_t\geq (\nu^2-\eta)\E\Big[\int_0^t |R^{1,2}_{s,s}|^{\nu-2} \dd s \Big]
    -L_{\eta} t= ((4-2\gamma)^2-\eta)\E\Big[\int_0^t |R^{1,2}_{s,s}|^{2(1-\gamma)} \dd s \Big]
    -L_{\eta} t.
\end{equation} 
Next, recalling \eqref{estib} and using that $|\nabla_x\psi(x,y)| \leq \nu |x-y|^{\nu-1}
=(4-2\gamma)|x-y|^{3-2\gamma}$, 
$$
|J^3_t| \leq L' \E\Big[\int_0^t |R^{1,2}_{s,s}|^{3-2\gamma} \frac{\dd s}{s^{\frac12+\frac1p}}\Big]
$$
for some constant $L'=L'(c_0,p,\chi,\theta,\gamma,\alpha)$. By the Young inequality with
$p=\frac{2(\gamma-1)}{2\gamma-3}$ and $p'=2(\gamma-1)$, we see that for all $\zeta>0$,
there is a constant $L'_\zeta=L'_\zeta(c_0,p,\chi,\theta,\gamma,\alpha)$ such that
\begin{equation}\label{ettac}
|J^3_t| \leq \zeta \E\Big[\int_0^t |R^{1,2}_{s,s}|^{2(1-\gamma)}\dd s\Big]
+L'_\zeta \int_0^t \frac{\dd s}{s^{(\gamma-1)(1+\frac2p)}} = \zeta \E\Big[\int_0^t |R^{1,2}_{s,s}|^{2(1-\gamma)}\dd s\Big]
+L'_\zeta t^{r_p},
\end{equation}
recall that $r_p=1-(\gamma-1)(1+\frac2p)>0$.
Finally, by \eqref{eq:borneDrift} and since $|\nabla_x\psi(x,y)|\leq(4-2\gamma)|x-y|^{3-2\gamma}$,
\begin{align*}
|B^{1,j}_t| 
\leq& \chi(4-2\gamma) \frac{\sqrt\theta C_0\big(\frac{4\alpha}{\theta}\big)\kappa\big(\frac{1}{2}, \gamma-1\big)}{4\pi} 
\E\Big[\int_0^t (S_s^{1,j})^{\frac 1{2(\gamma-1)}} |R^{1,2}_{s,s}|^{3-2\gamma} \dd s\Big].
\end{align*}
By the H\"older inequality (both for $\E$ and $\int_0^t$) with 
$p=2(\gamma-1)$ and $p'=\frac{2(\gamma-1)}{2\gamma-3}$,
\begin{align*}
|B^{1,j}_t| \leq& \chi  \frac{(4-2\gamma)\sqrt\theta C_0\big(\frac{4\alpha}{\theta}\big)\kappa\big(\frac{1}{2}, \gamma-1\big)}{4\pi} 
\Big(\E\Big[\int_0^t S_s^{1,j} \dd s \Big]\Big)^{\frac 1{2(\gamma-1)}} \Big(\E\Big[\int_0^t |R^{1,2}_{s,s}|^{2(1-\gamma)} \dd s\Big]\Big)^{\frac{2\gamma-3}{2(\gamma-1)}}\\
\leq & \chi  \frac{(4\!-\!2\gamma)\sqrt\theta C_0\big(\frac{4\alpha}{\theta}\big)\kappa\big(\frac{1}{2}, \gamma\!-\!1\big)}{4\pi} 
\Big(\frac{(1+\eta)\alpha^{1-\gamma}}{C_1(\alpha,\gamma) \!-\!\chi C_2(\theta,\alpha,\gamma)} \E\Big[\!\int_0^t \!|R_{u,u}^{1,2}|^{2(1-\gamma)}\dd u \Big] + A_\eta t^{r_p}   \Big)^{\frac 1{2(\gamma-1)}} \\
& \hskip7cm \times\Big(\E\Big[\int_0^t |R^{1,2}_{s,s}|^{2(1-\gamma)} \dd s\Big]\Big)^{\frac{2\gamma-3}{2(\gamma-1)}}
\end{align*}
for any $\eta>0$, by \eqref{eq:borneSt}
(with $A_\eta=A_\eta(c_0,p,\chi,\theta,\gamma,\alpha)$). 
Since $\frac1{2(\gamma-1)}<1$, allowing $A_\eta$ to change 
from line to line,
\begin{align*}
|B^{1,j}_t| \leq &\chi(1+\eta)^{\frac1{2(\gamma-1)}}\frac{(4-2\gamma)  \sqrt\theta C_0\big(\frac{4\alpha}{\theta}\big) \kappa\big(\frac{1}{2}, \gamma-1\big)}{4\pi\sqrt{\alpha}[C_1(\alpha,\gamma) -\chi C_2(\theta,\alpha,\gamma)]^{\frac 1 {2(\gamma-1)}}} 
\E\Big[\int_0^t |R^{1,2}_{s,s}|^{2(1-\gamma)} \dd s \Big]\\
&+A_\eta t^{\frac{r_p}{2(\gamma-1)}}\Big(\E\Big[\int_0^t |R^{1,2}_{s,s}|^{2(1-\gamma)} \dd s\Big]\Big)^{\frac{2\gamma-3}{2(\gamma-1)}}.
\end{align*}
We easily deduce, again by the Young inequality with
$p=2(\gamma-1)$ and $p'=\frac{2(\gamma-1)}{2\gamma-3}$, that for all $\zeta>0$,
there is a constant $L''_\zeta=L''_\zeta(c_0,\chi,\theta,\gamma,\alpha)$ such that
\begin{equation}\label{eq:ccc}
|B^{1,j}_t| \leq \Big(\chi\frac{(4-2\gamma)\sqrt\theta C_0\big(\frac{4\alpha}{\theta}\big)\kappa\big(\frac{1}{2}, \gamma-1\big)}{4\pi\sqrt\alpha[C_1(\alpha,\gamma) -\chi C_2(\theta,\alpha,\gamma)]^{\frac1{2(\gamma-1)}}}+\zeta\Big) 
\E\Big[\int_0^t |R^{1,2}_{s,s}|^{2(1-\gamma)} \dd s \Big]
+L''_\zeta t^{r_p}.
\end{equation}
Plugging \eqref{eq:aaa}, \eqref{eq:bbb}, \eqref{ettac} and \eqref{eq:ccc} in \eqref{eq:plug}, we have proved that for all $\eta>0$, all $\zeta>0$,
\begin{align*}1\geq &\Big( 2(4-2\gamma)^2-2\chi \frac{(4-2\gamma)  \sqrt\theta C_0\big(\frac{4\alpha}{\theta}\big)\kappa\Big(\frac{1}{2}, \gamma-1\Big)}{4\pi\sqrt\alpha[C_1(\alpha,\gamma) -\chi C_2(\theta,\alpha,\gamma)]^{\frac1{2(\gamma-1)}}}-2\eta-4\zeta\Big) \E\Big[\int_0^t |R^{1,2}_{s,s}|^{2(1-\gamma)} \dd s \Big]\\
&- 2 L_{\eta} t-2(L'_\zeta+L''_\zeta)t^{r_p}.
\end{align*}
By assumption, it is possible to find $\eta>0$ and $\zeta>0$ small enough
so that the constant in front of the first term of the right hand side is positive. Since this constant does not depend on $N$ nor on $\varepsilon$,
which is also the case of $L_{\eta}$, $L'_\zeta$ and $L''_\zeta$, 
this ends the proof of~\eqref{eq:Prop21-1}.

\medskip

Next, recalling that $S^{1,j}_t$ was defined in the statement of Proposition~\ref{lemma:LL}, we have
$$
\E\Big[\int_0^t \!\int_0^s \!\frac{1}{(s-u+|X^{1}_s-X^2_u|^2)^\gamma}\dd u\dd s\Big]\leq
\E\Big[\int_0^t \!\int_0^s \!\frac{1}{(s-u+\alpha|X^{1}_s-X^2_u|^2)^\gamma}\dd u\dd s\Big]=
\E\Big[\int_0^t S^{1,2}_s \dd s\Big],
$$
because $\alpha\in (0,1)$ (since $C_1(\alpha,\gamma)>0$ implies that 
$\alpha<\frac1{4(\gamma-1)}<\frac12$). Hence 
\eqref{eq:Prop21-2} directly follows from \eqref{eq:borneSt} (with e.g. $\eta=1$) and \eqref{eq:Prop21-1}. 

\medskip

By Remark \ref{ttt} with $\alpha=1$, \eqref{eq:Prop21-3prime} immediately follows from \eqref{eq:Prop21-2} and
\begin{align*}
&\E \Big[\int_0^t \Big| \int_0^s \nabla K^{\theta,\lambda}_{s-u}(X_s^{1}-X_u^{2}) \dd u \Big|^{2(\gamma-1)}  \dd s\Big ]\\
\leq&  \Big[\frac{\sqrt\theta C_0\big(\frac{4}{\theta}\big)}{4\pi}\Big]^{2(\gamma-1)}
\E \Big[\int_0^t \Big( \int_0^s   \frac{1}{(s-u+ |R^{1,2}_{s,u}|^2)^{\frac{3}{2}}}\dd u\Big)^{2(\gamma-1)}  \dd s\Big]\\
\leq & \Big[\frac{\sqrt\theta C_0\big(\frac{4}{\theta}\big)}{4\pi}\kappa\Big(\frac12,\gamma-1\Big)\Big]^{2(\gamma-1)}
\E \Big[\int_0^t  \int_0^s   \frac{1}{(s-u+ |R^{1,2}_{s,u}|^2)^{\gamma}}\dd u  \dd s\Big]
\end{align*}
by Lemma~\ref{lemma:FI} with $a=\frac12$ and $b=\gamma-1$. Hence \eqref{eq:Prop21-3} also follows from \eqref{eq:Prop21-2}.
\end{proof}

\section{Tightness}
\label{sec:tight}

Here we prove the tightness in $N\geq 2$ and $\e \in (0,1]$ of the smoothed particle system.
We closely follow \cite[Lemma~11]{fournier-jourdain}, although some additional moment conditions were assumed there.

\begin{lemma}\label{tight}
Consider some nonnegative $c_0 \in L^p(\R^2)$ for some $p>2$.
Let $\gamma\in (\frac32,\frac{2p+2}{p+2})$, $\alpha> 0$, $\chi>0$ and $ \theta >0$ satisfy the conditions of Proposition~\ref{prop:mc}. For each $N\geq 2$, each $\varepsilon\in(0,1]$, consider the unique solution $(X^{i,N,\e}_t)_{t\in [0,\infty),i=1,\dots,N}$
to \eqref{eq:PSS} with some exchangeable initial
condition $(X^{i,N}_0)_{i=1,\dots,N}$.

\medskip

(i) For $N \geq 2$ fixed, 
the family $((X^{1,N,\e}_t)_{t\geq 0}, \e\in(0,1])$ is tight 
in $C([0,\infty),\R^2)$. 

\medskip

(ii) If $(X^{1,N}_0)_{N\geq 2}$ is tight in $\R^2$, then 
$((X^{1,N,\e}_t)_{t\geq 0}, N\geq 2, \e\in(0,1])$ 
is tight in $C([0,\infty),\R^2)$.
\end{lemma}

\begin{proof} 
We start with (ii).
The space $C([0,\infty),\R^2)$ being endowed with the uniform convergence on compact time intervals, we only have to check that
$((X^{1,N,\e}_t)_{t\in [0,T]}, N\geq 2, \e\in(0,1] )$
is tight in $C([0,T],\R^2)$ for any $T>0$.
By definition, $X^{1,N,\e}_t=X^{1,N}_0 + \sqrt 2 W^1_t + 
\chi G^{1,N,\e}_t+ \chi \Gamma^{1,N,\e}_t$, where
\begin{gather*}
G^{1,N,\e}_t:= \int_0^t \nabla b^{c_0,\theta,\lambda}_{s+\e}(X^{1,N,\e}_s)\dd s ,\\
\Gamma^{1,N,\e}_t:=\frac{1}{N-1} \sum_{j=2}^N \int_0^t D^{1,j,N,\e}_s \dd s, \quad \hbox{with}\quad
D^{1,j,N,\e}_s :=\int_0^s H^{\theta,\lambda, \varepsilon}_{s-u} ( X^{1,N,\varepsilon}_s-X^{j,N,\varepsilon}_u)  \dd u.
\end{gather*}
The family $(X^{1,N}_0)_{N\geq 2}$ is tight by hypothesis and $(W^1_t)_{t\in[0,T]}$ does not depend on $N\geq 2$. The family 
$((G^{1,N,\e}_t)_{t\in [0,T]}, N\geq 2, \e\in(0,1])$ is tight because
by \eqref{estib}, it a.s. takes values in the set $\Kk$ of 
functions $x:[0,T]\mapsto \R^2$ such that $x(0)=0$ and for all $0\leq s < t \leq T$, $|x(t)-x(s)|\leq A |t-s|^{\frac12-\frac1p}$
(for some constant $A=A(c_0,p,T)$), and because $\Kk$ is compact in $C([0,T],\R^2)$ by Ascoli's theorem.

\medskip
It remains to  prove that the family $((\Gamma^{1,N,\e}_t)_{t\in [0,T]}, N\geq 2, \e\in(0,1])$ is tight in
$C([0,T],\R^2)$. 
Let $0\leq s < t \leq T$. We use H\"older's inequality with $p=\frac{2(\gamma-1)}{2\gamma-3}$
and $p'=2(\gamma-1)$ to get
\begin{align*}
\Big |\Gamma^{1,N,\e}_t-\Gamma^{1,N,\e}_s|  \leq  \frac{\chi}{N-1} \sum_{j=2}^N \int_s^t |D^{1,j,N,\e}_u| \dd u 
\leq |t-s|^{\frac{2\gamma-3}{2(\gamma-1)}} \frac{\chi}{ N-1} \sum_{j=2}^N \Big(\int_s^t |D^{1,j,N,\e}_u|^{2(\gamma-1)} \dd u\Big)^{\frac1{2(\gamma -1)}}. 
\end{align*}
Setting $\beta=\frac{2\gamma-3}{2(\gamma-1)}>0$ and using that $x^{\frac 1 {2(\gamma-1)}}\leq 1 + x$ (because $2(\gamma-1)>1$), we conclude that
\begin{align*}
\Big |\Gamma^{1,N,\e}_t-\Gamma^{1,N,\e}_s|  \leq Z_T^{N,\e} |t-s|^{\beta},
\quad \hbox{where}\quad  Z_T^{N,\e}:=\frac\chi {N-1} \sum_{j=2}^N \Big[1+
\int_0^T  |D^{1,j,N,\e}_u|^{2(\gamma-1)} \dd u\Big].
\end{align*}
Since $|H^{\theta,\lambda,\e}_s(x)|\leq |\nabla K^{\theta,\lambda}_s(x)|$,
we have by \eqref{eq:Prop21-3} and exchangeability
$$
C_T:=\!\!\!\!\sup_{\e\in(0,1],N\geq 2} \E[Z^{N,\e}_T] \leq
\chi \sup_{\e\in(0,1],N\geq 2} \E\Big[1+\int_0^T \Big(\int_0^u |\nabla K^{\theta,\lambda}_{u-s}(X^{1,N,\varepsilon}_s-X^{j,N,\varepsilon}_u)| \dd s \Big)^{2(\gamma-1)}\dd u\Big]
< \infty.
$$
Now, let $\Kk'_{M}$  the set of functions $x:[0,T]\mapsto \R^2$ such that $x(0)=0$ and for all $0\leq s < t \leq T$, $|x(t)-x(s)|\leq M |t-s|^\beta$.
For all $\e\in(0,1]$, all $N\geq2$ and all $M>0$, 
$$
\P((\Gamma^{1,N,\e}_t)_{t\in [0,T]} \notin
\Kk'_{M})\leq \P(Z^{N,\e}_T > M) \leq \frac{C_T}M.
$$ 
Since $\Kk'_{M}$ is compact in $C([0,T],\R^2)$ by Ascoli's theorem,
the proof of (ii) is complete. The proof of (i) is the same,
but we do not need the tightness of the family
$(X^{1,N}_0)_{N\geq 2}$ since $N\geq 2$ is fixed.
\end{proof}
\section{Existence of the particle system}
\label{sec:ex}

Here we show the existence of the particle system without cutoff. We follow the ideas of \cite[Theorem~5]{fournier-jourdain} and combine them with our results from Section~\ref{sec:main}.

\begin{proposition}\label{prop:ex}
Consider some nonnegative $c_0 \in L^p(\R^2)$ for some $p>2$. Let $\gamma\in(\frac{3}{2}, \frac{2p+2}{p+2})$, $\alpha>0$, $\chi>0$ and $\theta>0$ satisfy the conditions of 
Proposition~\ref{prop:mc}. Fix $N\geq 2$ and consider
some exchangeable initial condition $(X_0^{i,N})_{i=1,\dots,N}$.
There exists a (weak) solution $(X^{i,N}_t)_{t\geq 0,i=1,\dots,N}$  to \eqref{def:PS}. Moreover, the family $((X^{i,N}_t)_{t\geq 0},i=1,\dots,N)$ 
is exchangeable, and for all $t>0$,
\begin{gather}
\label{eq:Prop31-1}
    \sup_{N\geq 2} \E \Big[\int_0^t  \frac{1}{|X_s^{1,N}-X_s^{2,N}|^{2(\gamma-1)}} \dd s\Big ]<\infty,\\
    \label{eq:Prop31-2}
    \sup_{N\geq 2} \E \Big[\int_0^t  \int_0^s \frac{1}{(s-u+ |X_s^{1,N}-X_u^{2,N}|^{2})^\gamma} \dd u \dd s \Big ]<\infty,\\
    \label{eq:Prop31-3prime}
    \sup_{ N\geq 2} \E \Big[\int_0^t \int_0^s |\nabla K^{\theta,\lambda}_{s-u}(X_s^{1,N}-X_u^{2,N})|^{\frac{2\gamma} 3} \dd u   \dd s\Big ]<\infty,\\
    \label{eq:Prop31-3}
     \sup_{N\geq 2} \E \Big[\int_0^t \Big( \int_0^s |\nabla K^{\theta,\lambda}_{s-u}(X_s^{1,N}-X_u^{2,N})| \dd u \Big)^{2(\gamma-1)}  \dd s\Big ]<\infty.
\end{gather}
\end{proposition}

\begin{proof} The only difference w.r.t. the proof of \cite[Theorem~5]{fournier-jourdain} lies in the last step.

\medskip

{\it Step 1.} For each $\e\in(0,1]$, let  $(X^{i,N,\e}_t)_{t\in [0,\infty),i=1,\dots,N}$ solve
to \eqref{eq:PSS}. By Lemma~\ref{tight}-(i), we know that the family
$((X^{1,N,\e}_t)_{t\geq 0}, \e\in(0,1])$ is tight in $C([0,\infty),\R^2)$. By exchangeability, the family
$((X^{1,N,\e}_t,\dots,X^{N,N,\e}_t)_{t\geq 0},  \e\in(0,1])$ is tight in $C([0,\infty),(\R^2)^N)$ and
consequently, the family $(((X^{1,N,\e}_t,W^{1}_t),\dots,(X^{N,N,\e}_t,W^N_t))_{t\geq 0},  \e\in(0,1])$ is tight in 
$C([0,\infty),(\R^2\times\R^2)^{N})$.
Hence, there exists a decreasing sequence $\e_k \to 0$ such that
$((X^{1,N,\e_k}_t,W^{1}_t),\dots,(X^{N,N,\e_k}_t,W^N_t))_{t\geq 0}$ converges in law in
$C([0,\infty),(\R^2\times\R^2)^{N})$ as $k \to \infty$. Applying the Skorokhod representation theorem, we can find, for each $k\geq 1$,
a solution $(\tX^{1,N,\e_k}_t,\dots,\tX^{N,N,\e_k}_t)_{t\geq 0}$ to \eqref{eq:PSS}, associated to some Brownian
motions $(\tW^{1,N,\e_k}_t,\dots,\tW^{N,N,\e_k}_t)_{t\geq 0}$, in such a way that 
$((\tX^{1,N,\e_k}_t,\tW^{1,N,\e_k}_t),\dots,(\tX^{N,N,\e_k}_t,\tW^{N,N,\e_k}_t))_{t\geq 0}$ a.s. goes to some limit
$((X^{1,N}_t,W^{1}_t),\dots,(X^{N,N}_t,W^{N}_t))_{t\geq 0}$, as $k\to \infty$, in $C([0,\infty),(\R^2\times\R^2)^{N})$.
Of course, $((X^{i,N}_t)_{t\geq 0}, i=1,\dots,N)$ is exchangeable and we deduce \eqref{eq:Prop31-1}-\eqref{eq:Prop31-3}
from \eqref{eq:Prop21-1}-\eqref{eq:Prop21-3} and the Fatou Lemma.
Observe that \eqref{PScond} follows from \eqref{eq:Prop31-3prime} (by exchangeability) since $\gamma>3/2$.

\medskip

{\it Step 2.} We introduce $\Ff_t=\sigma((X_s^{i,N},W^{i}_s)_{i=1,\dots,N,s\in[0,t]})$.
Of course, $(X^{i,N}_t)_{i=1,\dots,N,t\geq 0}$ is $(\Ff_t)_{t\geq 0}$-adapted.  
Exactly as in Step 2 of the proof of \cite[Theorem~5]{fournier-jourdain}, one can show that
$(W^{i}_t)_{i=1,\dots,N,t\geq 0}$ is a $2N$-dimensional $(\Ff_t)_{t\geq 0}$-Brownian motion.

\medskip

{\it Step 3.}
It only remains to check that for each $i\in \{1,\dots,N\}$, each $t\geq 0$,
$$
X^{i,N}_t=X^{i,N}_0 + \sqrt 2 W^{i}_t + \chi Y^{i,N}_t+ \frac \chi {N-1} \sum_{j\neq i} Z^{i,j,N}_t,$$
where
$$
Y^{i,N}_t=\int_0^t \nabla b^{c_0,\theta,\lambda}_s(X^{i,N}_s)\dd s
\quad \hbox{and}\quad Z^{i,j,N}_t=
\int_0^t \int_0^s \nabla K_{s-u}^{\theta,\lambda}(X^{i,N}_s-X^{j,N}_u) \dd u \dd s.
$$
We start from 
$\tX^{i,N,\e_k}_t=\tX^{i,N,\e_k}_0 + \sqrt 2 \tW^{i,N,\e_k}_t + \chi Y^{i,N,\e_k}_t
+ \frac\chi {N-1} \sum_{j\neq i} Z^{i,j,N,\e_k}_t$,
where 
$$Y^{i,j,N,\e}_t=\int_0^t \nabla b^{c_0,\theta,\lambda}_{s+\e}(\tX^{i,N,\e_k}_s)\dd s
\quad \hbox{and}\quad Z^{i,j,N,\e_k}_t=
\int_0^t \int_0^s H^{\theta,\lambda,\e}_{s-u}(\tX^{i,N,\e}_s-\tX^{j,N,\e}_u) \dd u
$$
and pass to the limit as $k\to \infty$, e.g. in probability. 
Of course, $(\tX^{i,N,\e_k}_t,\tX^{i,N,\e_k}_0,\tW^{i,N,\e_k}_t)$ a.s. tends to $(X^{i,N}_t,X^{i,N}_0,W^i_t)$ by construction, 
and $Y^{i,N,\e_k}_t$ a.s. tends to $Y^{i,N}_t$ by dominated convergence,
recalling \eqref{estib} and noting that a.s., $b^{c_0,\theta,\lambda}_{s+\e}(\tX^{i,N,\e_k}_s)$ tends to $b^{c_0,\theta,\lambda}_{s}(X^{i,N}_s)$ for all $s>0$
(because $(s,x)\mapsto \nabla b_s^{c_0,\theta,\lambda}(x)$ is continuous on $(0,\infty)\times\R^2$).

\medskip
It remains to show that $Z_t^{i,j, N, \e_k}\to  Z_t^{i,j, N}$ in probability as $k\to\infty$. We fix $\eta>0$ and decompose
\begin{align*}
    |Z_t^{i,j, N, \e_k}- Z_t^{i,j, N}|\leq & \int_0^t \int_0^s |H_{s-u}^{\theta,\lambda,\e_k} (\tX_s^{i,N,\e_k}-\tX_u^{j,N,\e_k})- H_{s-u}^{\theta,\lambda,\eta} (\tX_s^{i,N,\e_k}-\tX_u^{j,N,\e_k}) |\dd u \dd s \\
    & +\int_0^t \int_0^s |H_{s-u}^{\theta,\lambda,\eta} (\tX_s^{i,N,\e_k}-\tX_u^{j,N,\e_k})- H_{s-u}^{\theta,\lambda,\eta} (X_s^{i,N}-X_u^{j,N}) |\dd u \dd s \\
    &+\int_0^t \int_0^s |H_{s-u}^{\theta,\lambda,\eta} (X_s^{i,N}-\tX_u^{j,N})- \nabla K_{s-u}^{\theta,\lambda} (X_s^{i,N}-X_u^{j,N}) |\dd u \dd s \\
    =:& I^{i,j}_{1,k,\eta,t}+I^{i,j}_{2,k,\eta,t}+I^{i,j}_{3,\eta,t}.
 \end{align*}
 First, for each $\eta>0$,
 $\lim_{k\to \infty}I^{i,j}_{2,k,\eta,t}=0$ a.s.,  by dominated convergence, because $x\mapsto H^{\theta,\lambda,\eta}_s(x)$ is continuous and uniformly bounded and since
 $(\tX_s^{i,N,\e_k},\tX_u^{j,N,\e_k}) \to (X_s^{i,N},X_u^{j,N})$ a.s.
 
\medskip 
We now check that
 \begin{equation}\label{toprove}
\lim_{\eta\to 0}  \limsup_{k\to \infty}   \E[I^{i,j}_{1,k,\eta,t}+I^{i,j}_{3,\eta,t}]=0,
\end{equation}
and this will complete the proof.
Recalling that $H_{s}^{\theta,\lambda,\e}=\frac{s^2}{(s+\e)^2}\nabla K_{s}^{\theta,\lambda,\e}$, one verifies that, if $k$ is large enough so that $\e_k \in (0,\eta]$, since $\frac{s^2}{(s+\eta)^2}\leq \frac{s^2}{(s+\e_k)^2}\leq 1$,
$$
I^{i,j}_{1,k,\eta,t}+I^{i,j}_{3,\eta,t}\leq \int_0^t \!\! \int_0^s \!\!
\Big(1-\frac{(s-u)^2}{(s-u+\eta)^2} \Big)
\Big(|\nabla K_{s-u}^{\theta,\lambda} (\tX_s^{i,N,\e_k}-\tX_u^{j,N,\e_k})|+ |\nabla K_{s-u}^{\theta,\lambda} (X_s^{i,N}-X_u^{j,N}) |\Big) \dd u \dd s. 
$$
Applying H\"older's inequality with $p=\frac{2\gamma}{2\gamma-3}$ and $p'=\frac{2\gamma}3$, we find
\begin{align*}
    I^{i,j}_{1,k,\eta,t}+I^{i,j}_{3,\eta,t}\leq& \Big(\int_0^t  \int_0^s \Big(1-\frac{(s-u)^2}{(s-u+\eta)^2} \Big)^{\frac{2\gamma} {2\gamma-3}}\dd u \dd s\Big)^{\frac{2\gamma-3}{2\gamma}}\\
    & \times \Big(\int_0^t  \int_0^s\Big(|\nabla K_{s-u}^{\theta,\lambda} (\tX_s^{i,N,\e_k}-\tX_u^{j,N,\e_k})|+ |\nabla K_{s-u}^{\theta,\lambda} (X_s^{i,N}-X_u^{j,N}) |\Big)^{\frac{2\gamma}3} \dd u \dd s\Big)^{\frac3{2\gamma}}.
\end{align*}
By \eqref{eq:Prop21-3prime} and \eqref{eq:Prop31-3prime} (and since $\frac3{2\gamma}<1$), we deduce that for some constant $C>0$,
$$
\limsup_{k\to \infty} \E[I^{i,j}_{1,k,\eta,t}+I^{i,j}_{3,\eta,t}]\leq C \Big(\int_0^t  \int_0^s \Big(1-\frac{(s-u)^2}{(s-u+\eta)^2} \Big)^{\frac{2\gamma} {2\gamma-3}}\dd u \dd s\Big)^{\frac{2\gamma-3}{2\gamma}},
$$
which tends to $0$ as $\eta\to 0$ by dominated convergence. This proves \eqref{toprove}.
\end{proof}

\section{Convergence}
\label{sec:conc}

We prove that the empirical measure of the particle system converges, up to extraction of a subsequence,
to a solution of the martingale problem. We recall that $\Pp(\R^2)$ and $\Pp(C([0,\infty),\R^2))$ 
are endowed with their weak convergence topologies.

\begin{theorem}\label{conv}
Consider some nonnegative $c_0 \in L^p(\R^2)$ for some $p>2$. Let $\gamma\in(\frac{3}{2}, \frac{2p+2}{p+2})$, $\alpha>0$, $\chi>0$ and $\theta>0$ satisfy the conditions of 
Proposition~\ref{prop:mc}. 
Consider,
for each $N\geq 2$, the particle system 
$(X^{i,N}_t)_{t\geq 0,i=1,\dots,N}$ built in Proposition~\ref{prop:ex}, as well as $\mu^N=N^{-1}\sum_1^N \delta_{(X^{i,N}_t)_{t\geq 0}}$, 
which a.s. belongs to $\Pp(C([0,\infty),\R^2))$. For each $t\geq 0$, we set
$\mu^N_t=N^{-1}\sum_1^N \delta_{X^{i,N}_t}$, which a.s. belongs to $\Pp(\R^2)$.
We assume that $\mu^N_0$ converges in probability, as $N\to \infty$, to some $\rho_0 \in \Pp(\R^2)$.

\medskip

The family $(\mu^N, N\geq 2)$ is tight in $\Pp(C([0,\infty),\R^d))$ and any (possibly random) limit point $\mu$ of  $(\mu^N)_{N\geq 2}$ a.s. solves \hyperref[defMP]{$\mathcal{(MP)}$} with initial law $\rho_0$.
Moreover, for $(\mu_t)_{t\geqq 0}$ its family of time marginals,
for all $t\geq 0$,
\begin{gather}
\label{eq:ConditionMPplus}
\E\Big[\int_0^t \int_{\R^2} \int_0^s \int_{\R^2}   (|K^{\theta,\lambda}_{s-u}(x-y)|^\gamma + |\nabla K^{\theta,\lambda}_{s-u}(x-y)|^{\frac{2\gamma}3}) \mu_u(\dd y)  \dd u   \mu_s(\dd x) \dd s \Big]< \infty,\\
\E\Big[\int_0^t \int_{\R^2} \int_{\R^2}  \frac{1}{|x-y|^{2(\gamma-1)}} \mu_s(\dd y)\mu_s(\dd x) \dd s \Big]< \infty.\label{eq:ConditionMPplus2}
\end{gather}
\end{theorem}

\begin{proof}
Again, we follow closely the proof of \cite[Theorem 6]{fournier-jourdain}. During this proof, we use the shortened notation $C=C([0,\infty),\R^2)$.

\medskip
{\it Step 1.}
For each $N\geq 2$, $(X^{i,N}_t)_{t\geq 0,i=1,\dots,N}$ has been built as a limit point
of $(X^{i,N,\e}_t)_{t\in [0,\infty),i=1,\dots,N}$ as $\e\to 0$. By 
Lemma~\ref{tight}-(ii), the family $((X^{1,N}_t)_{t\geq 0}, N\geq 2)$ is thus tight in $C$.
Since the system is exchangeable, this implies, see Sznitman \cite[Proposition~2.2]{Sznitman}, that
the family $(\mu^{N}, N\geq 2)$ is tight in $\Pp(C)$. We now consider a (non relabelled) subsequence of $\mu^N$ that converges in law to some $\mu$ as $N\to \infty$. We denote by $(\mu_t)_{t\geq 0}$ its family of time-marginals.
Since $\mu^N_0$ goes to $\rho_0$ by assumption,
we have $\mu_0=\rho_0$ a.s. 

\medskip

Moreover, since $\mu^N$ converges
in law to $\mu$, it also holds true that $\mu^N \otimes \mu^N$
and $\mu^N \odot \mu^N$ both converge in law to $\mu\otimes \mu$
in $\Pp(C\times C)$,
where we have set 
$$
\mu^N \odot \mu^N=
\frac 1 {N(N-1)}\sum_{i\neq j} \delta_{((X^{i,N}_t)_{t\geq 0},(X^{j,N}_t)_{t\geq 0})}. 
$$
Hence we deduce from the Fatou lemma  that for all $t\geq 0$,
\begin{align*}
\E\Big[\int_0^t \int_0^s \int_{\R^2} \int_{\R^2}  &\frac {\mu_u(\dd y)\mu_s(\dd x) \dd u \dd s}{(|{s-u}+(x-y)^2|)^\gamma}\Big]
= \E\Big[\int_C\int_C\int_0^t \int_0^s  \frac { \dd u \dd s}{(|{s-u}+(x_s-y_u)^2|)^\gamma}(\mu\otimes \mu)(\dd x,\dd y)\Big]\\
\leq & \liminf_N \E\Big[\int_C\int_C\int_0^t \int_0^s \frac { \dd u  \dd s}{(|{s-u}+(x_s-y_u)^2|)^\gamma}(\mu^N \odot \mu^N) (\dd x,\dd y) \Big]\\
=& \liminf_N \frac 1{N(N-1)}\sum_{i\neq j}\E\Big[\int_0^t \int_0^s \frac {\dd u  \dd s}{(|{s-u}+(X^{i,N}_s-X^{j,N}_u)^2|)^\gamma}\Big],
\end{align*}
which is finite by \eqref{eq:Prop31-2} and exchangeability.
Moreover, there is a constant $A=A(\theta)>0$ such that
$$
K^{\theta,\lambda}_s(x) \leq \frac 1 {4\pi s}e^{-\frac{\theta}{4s}|x|^2}\leq \frac{A}{s+|x|^2}
\quad \hbox{and} \quad |\nabla K^{\theta,\lambda}_s(x)| \leq
\frac{A}{(s+|x|^2)^{3/2}}.
$$
The first estimate easily follows from the fact $z\mapsto (1+z)e^{-z}$ is bounded on $\R_+$, and the second one has been shown in Remark \ref{ttt}.
We conclude that \eqref{eq:ConditionMPplus} holds true.
Observe that this implies \eqref{eq:ConditionMP} because $\gamma>3/2$.
Similarly, \eqref{eq:ConditionMPplus2} is deduced from \eqref{eq:Prop31-1} 
and the Fatou lemma.

\medskip

{\it Step 2.} It only remains to prove that a.s., for any $\varphi\in C^2_c(\R^2)$, the process $(M_t^\varphi)_{t\geq 0}$, defined in \eqref{def_mart}, is a $\mu$-martingale. To this end, it suffices to show that for all $t>s>0$, all continuous bounded function $\Phi:C\to \R$, we have $\Psi(\mu)=0$ a.s., where
for $\Q \in \Pp(C)$,
\begin{align*}
\Psi(\Q)=\int_{C}
\Phi((x_r)_{r\in [0,s]})
\Big(\varphi(x_t)-\varphi(x_s)-&
\int_s^t \Big[ \Delta \varphi(x_u)+ \chi \nabla \varphi(x_u)\cdot \nabla b^{c_0,\theta,\lambda}_s(x_u)\\
&+\chi \nabla \varphi(x_u)\cdot  \int_0^u (\nabla K^{\theta,\lambda}_{u-v}\ast \Q_v)(x_u) \dd v \Big]\dd u \Big) \Q(\dd x).
\end{align*}
We observe that for any $\Q\in \Pp(C)$, it holds that $\Psi(\Q)=\Theta(\Q\otimes \Q)$, where for $\Pi\in \Pp(C\times C)$,
\begin{align*}
\Theta(\Pi)=\int_{C\times C}\!\!\!\!\!
\Phi((x_r)_{r\in [0,s]})
\Big(\varphi(x_t)-\varphi(x_s)-&
\int_s^t \Big[ \Delta \varphi(x_u) + \chi \nabla \varphi(x_u)\cdot \nabla b^{c_0,\theta,\lambda}_s(x_u)\\
&+\chi \nabla \varphi(x_u)\cdot  \int_0^u \nabla K^{\theta,\lambda}_{u-v}(x_u-y_v) \dd v \Big]\dd u \Big) \Pi(\dd x,\dd y).
\end{align*}

{\it Step 2.1.} Here we show that for some constant $A$, for all $N\geq 2$,
\begin{equation}\label{scc1}
\E \Big[ [\Theta(\mu^N\odot \mu^N)]^2  \Big] \leq \frac {A} N.
\end{equation}
We have
\begin{align*}
\Theta(\mu^N\odot \mu^N)=&\frac1{N}\sum_{i=1}^N
\Phi((X^{i,N}_r)_{r\in [0,s]}) 
\Big(\varphi(X^{i,N}_t)-\varphi(X^{i,N}_s) -\int_s^t \Big[ \Delta \varphi(X^{i,N}_u)\\
& \hskip0.5cm
+\chi \nabla \varphi(X^{i,N}_u)\cdot \nabla b^{c_0,\theta,\lambda}_s(X^{i,N}_u)+ \frac{\chi}{N-1}\sum_{j\neq i}\int_0^u \nabla K^{\theta,\lambda}_{u-v}(X^{i,N}_u-X^{j,N}_v) \dd v \Big]\dd u \Big)\\
=&\frac 1N\sum_{i=1}^N
\Phi((X^{i,N}_r)_{r\in [0,s]})(O^{i,N}_t-O^{i,N}_s),
\end{align*}
where
\begin{align*}
O^{i,N}_t:=&\varphi(X^{i,N}_t) - \int_0^t \Delta \varphi(X^{i,N}_s)\dd s
-\chi \int_0^t \nabla\varphi(X^{i,N}_s) \cdot \nabla b^{c_0,\theta,\lambda}_s(X^{i,N}_s)\dd s\\
&\hskip1.5cm-\frac \chi {N-1} \sum_{j\neq i} \int_0^t \nabla\varphi(X^{i,N}_s) \cdot \Big(\int_0^u \nabla K^{\theta,\lambda}_{u-v}(X^{i,N}_u-X^{j,N}_v) \dd v \Big)\dd s \\
=& \varphi(X^{i,N}_0) + \sqrt 2 \int_0^t \nabla \varphi(X^{i,N}_s) \cdot \dd W^i_s
\end{align*}
by the It\^o formula (starting from \eqref{def:PS}).
Then \eqref{scc1} follows from some easy stochastic calculus arguments, because
$\Phi$ and $\nabla \varphi$ are bounded and since the Brownian motions $(W^1_t)_{t\geq 0},\dots,(W^N_t)_{t\geq 0}$ are independent.

\medskip

{\it Step 2.2.} Next we introduce, for $\eta\in (0,1]$, 
$\Theta_\eta$ defined as $\Theta$ with $\nabla K^{\theta,\lambda}_s(x)$ replaced by the smooth and bounded kernel $H^{\theta,\lambda,\eta}_s(x)=\frac{s^2}{(s+\eta)^2}\nabla K^{\theta,\lambda}_s(x)$, recall \eqref{Heps}.
Then one easily checks that the map $\Pi \mapsto \Theta_\eta(\Pi)$
is continuous and bounded from $\Pp(C\times C)$ to $\R$. 
This uses in particular \eqref{estib} and that $x\mapsto \nabla b^{c_0,\theta,\lambda}_s(x)$ is continuous for all $s>0$.
Since
$\mu^N$ goes in law to $\mu$ and thus, as already seen, $\mu^N\odot\mu^N$ goes in law to $\mu \otimes \mu$, we deduce that for any $\eta\in(0,1]$,
$$
\E[|\Theta_\eta(\mu\otimes \mu)|]=\lim_{N} \E[|\Theta_\eta(\mu^N\odot\mu^N)|].
$$

{\it Step 2.3.} We now prove that $\lim_{\eta\to 0} \Delta_\eta=0$, where
$$
\Delta_\eta:= \E[|\Theta(\mu\otimes \mu)-\Theta_\eta(\mu\otimes \mu)|] + \sup_{N\geq 2}\E[|\Theta(\mu^N\odot\mu^N)-\Theta_\eta(\mu^N\odot\mu^N)|].
$$
We proceed as in the proof of \eqref{toprove}.
Since $\Phi$ and $\nabla \varphi$ are bounded, we see that for some constant $A$, for any $\Pi \in \Pp(\R^2\times\R^2)$,
\begin{align*}
|\Theta(\Pi)-\Theta_\eta(\Pi)|\leq&  A \int_{C\times C}
\int_0^t \int_0^u |\nabla K^{\theta,\lambda}_{u-v}(x_u-y_v)-H^{\theta,\lambda,\eta}_{u-v}(x_u-y_v)| \dd v \dd u \Pi(\dd x,\dd y)\\
=&  A \int_{C\times C}
\int_0^t \int_0^u \Big(1-\frac{(s-u)^2}{(s-u+\eta)^2}\Big)|\nabla K^{\theta,\lambda}_{u-v}(x_u-y_v)| \dd v \dd u \Pi(\dd x,\dd y)\\
\leq& A \e_\eta \int_{C\times C} \Big(\int_0^t  \int_0^s |\nabla K_{s-u}^{\theta,\lambda}(x_u-y_v) |^{\frac{2\gamma}3} \dd u \dd s\Big)^{\frac3{2\gamma}} \Pi(\dd x,\dd y)
\end{align*}
by the H\"older inequality, with 
$\e_\eta=(\int_0^t  \int_0^s (1-\frac{(s-u)^2}{(s-u+\eta)^2})^{\frac{2\gamma} {2\gamma-3}}\dd u \dd s)^{\frac{2\gamma-3}{2\gamma}}$. Hence
\begin{align*}
\Delta_\eta\leq & A\e_\eta\E\Big[ \int_{C\times C} \Big(\int_0^t  \int_0^s |\nabla K_{s-u}^{\theta,\lambda}(x_u-y_v) |^{\frac{2\gamma}3} \dd u \dd s\Big)^{\frac3{2\gamma}} (\mu\otimes\mu+\mu^N\odot\mu^N)(\dd x,\dd y)\Big]\\
=& A\e_\eta \E\Big[\Big(\int_0^t  \int_{\R^2}\int_0^s\int_{\R^2} |\nabla K_{s-u}^{\theta,\lambda}(x-y) |^{\frac{2\gamma}3} \mu_u(\dd y)\dd u \mu_s(\dd x)\dd s\Big)^{\frac3{2\gamma}}\Big]\\
&+A\e_\eta\E\Big[\Big(\int_0^t \int_0^s\Big( \frac{1}{N(N-1)}\sum_{i\neq j} |\nabla K_{s-u}^{\theta,\lambda}(X^{i,N}_s-X^{j,N}_u)|^{\frac{2\gamma}3} \dd u \dd s
\Big)^{\frac3{2\gamma}}\Big],
\end{align*}
Since $\lim_{\eta\to0}\e_\eta=0$ by dominated convergence,
the conclusion follows from \eqref{eq:ConditionMPplus}, \eqref{eq:Prop31-3prime} and exchangeability (recall that $\frac 3{2\gamma}<1$).

\medskip
{\it Step 2.4.} Recalling that $\Psi(\mu)=\Theta(\mu\otimes\mu)$, we may write, 
for any $\eta \in (0,1]$,
\begin{align*}
\E[|\Psi(\mu)|]
\leq & \E[|\Theta(\mu\otimes\mu)-\Theta_\eta(\mu\otimes \mu)|]+ \limsup_N| \E[|\Theta_\eta(\mu\otimes \mu)|]- \E[|\Theta_\eta(\mu^N\odot \mu^N)|]|\\
&+ \limsup_N \E[|\Theta_\eta (\mu^N\odot\mu^N)-\Theta(\mu^N\odot\mu^N)|] + \limsup_N \E[|\Theta(\mu^N\odot \mu^N)|].
\end{align*}
The last term is equal to $0$ by Step 2.1, as well as the second one by Step 2.2. 
Hence
$$
\E[|\Psi(\mu)|] \leq \E[|\Theta(\mu\otimes\mu)-\Theta_\eta(\mu\otimes \mu)|]+ 
\limsup_N \E[|\Theta_\eta (\mu^N\odot\mu^N)-\Theta(\mu^N\odot\mu^N)|] .
$$
Step 2.3 thus implies that $\E[|\Psi(\mu)|]=0$, which was our goal.
\end{proof}

\section{Conclusion and discussion about the constants}
\label{sec:condition}
Recall that for $\alpha,\beta,\theta>0$ and $\gamma\in (3/2,2)$,
\begin{gather*}
C_0(\beta):=\sup_{u\geq 0} \sqrt u (1+\beta u) ^{3/2} e^{-u},
\qquad    C_1(\alpha,\gamma):=(\gamma-1)(1-4\alpha(\gamma-1)),\\
C_2(\theta,\alpha,\gamma):= \frac{\sqrt{\alpha\theta}(\gamma-1)}{2\pi}C_0\Big(\frac{4\alpha}{\theta}\Big) \kappa\Big(\frac{1}{2}, \gamma-1\Big)\kappa\Big(\gamma-\frac32,\gamma-1\Big).
\end{gather*}
Fix $\rho_0 \in \Pp(\R^2)$, $c_0 \in L^p(\R^2)$ for some $p>2$, $\chi>0$ and $\theta>0$. By Proposition~\ref{prop:ex} and
Theorem~\ref{conv}, the conclusions
of Theorem~\ref{th:mainTH} hold true provided there are 
$\gamma\in(\frac{3}{2}, \frac{2p+2}{p+2})$ and $\alpha>0$ such that 
$$ 
C_1(\alpha,\gamma)>0,\quad
C_1(\alpha,\gamma)>\chi C_2(\theta,\alpha,\gamma),
\quad  (4-2\gamma)-\chi \frac{\sqrt\theta C_0\big(\frac{4\alpha}\theta\big) \kappa\big(\frac{1}{2}, \gamma-1\big)}{4\pi\sqrt\alpha [C_1(\alpha,\gamma) -\chi C_2(\theta,\alpha,\gamma)]^{\frac1{2(\gamma-1)}}}>0.
$$
The first condition implies that $\alpha \in (0,\frac1{4(\gamma-1)})$ and the two other ones can be summarized as
$$
\chi  C_2(\theta,\alpha,\gamma) +[ \chi C_3(\theta,\alpha,\gamma)]^{2(\gamma-1)} < C_1(\alpha,\gamma),
$$
where we have set
$$
C_3(\theta,\alpha,\gamma):=\frac{\sqrt\theta C_0\big(\frac{4\alpha}\theta\big) \kappa\big(\frac{1}{2}, \gamma-1\big)}{4\pi\sqrt\alpha(4-2\gamma)}.
$$
Hence if we set, for $\theta>0$, $\gamma \in (\frac32,2)$ and $\alpha \in (0,\frac{1}{4(\gamma-1)})$, 
$$
\chi_{\theta,\alpha,\gamma}= \sup\Big\{\chi>0 :\chi C_2(\theta,\alpha,\gamma) + [\chi C_3(\theta,\alpha,\gamma)]^{2(\gamma-1)} < C_1(\alpha,\gamma) \Big\}, 
$$
Theorem~\ref{th:mainTH} holds true with 
\begin{equation}\label{Eq:chistar}
\chi_{\theta,p}^*=\sup\Big\{ \chi_{\theta,\alpha,\gamma} : \gamma \in \Big(\frac32,\frac{2p+2}{p+2}\Big), \alpha \in \Big(0,\frac1{4(\gamma-1)}\Big) \Big\}.
\end{equation}

We now discuss the numerical values of this threshold, obtained by numerical trials. We do not really take care of $p$: We try to find the values of $\gamma\in (\frac32,2)$ and $\alpha \in (0,\frac1{4(\gamma-1)})$ maximizing $\chi_{\theta,\alpha,\gamma}$ 
and then see to which values of $p>2$ this applies.

\begin{remark}\label{vn}
(i) For any $p>2$, $\liminf_{\theta \to 0} \chi_{\theta,p}^* \geq 3.28$.
\medskip

(ii) For any $p>2.6$, $\chi_{0.1,p}^*\geq 2.42$.
\medskip

(iii) For any $p>3.3$, $\chi_{1,p}^*\geq 1.39$.
\medskip

(iv) For any $p>3.5$, $\chi_{10,p}^*\geq 0.51$.
\medskip

(v) For any $p> 3.5$, $\liminf_{\theta\to \infty} \sqrt \theta\chi_{\theta,p}^* \geq 1.65$.
\end{remark}

\begin{proof}
We start with (i). Fix $p>2$. We choose $\gamma=\frac32+\sqrt{\theta}$, which belongs to $(\frac 32,\frac{2p+2}{p+2})$ for all $\theta>0$ small enough, and $\alpha=\sigma\,\theta$ (for some $\sigma>0$ to be chosen later)
which belongs to $(0,\frac1{4(\gamma-1)})$ for all $\theta>0$ small enough.
It holds that $\lim_{\theta\to 0} C_1(\alpha,\gamma)=\frac12$. Moreover,
$$
C_2(\theta,\alpha,\gamma)=\frac{\sqrt{\sigma}\,\theta (\frac12+\sqrt\theta)C_0(4\sigma)\kappa(\frac12,\frac12+\sqrt\theta),\kappa(\sqrt\theta,\frac12+\sqrt\theta)}{2\pi}\to 0
$$
as $\theta\to 0$, because $\kappa(\frac12,\frac12+\sqrt\theta)\to 1$ and $\kappa(\sqrt\theta,\frac12+\sqrt\theta)\sim \frac 1 {\sqrt \theta}$. Finally,
$$
C_3(\theta,\alpha,\gamma)=\frac{ C_0(4\sigma)\kappa(\frac12,\frac12+\sqrt\theta)}{\sqrt{\sigma} 4\pi(1-2\sqrt\theta)}
\to \frac{ C_0(4\sigma)}{4\pi\sqrt\sigma}.
$$
All in all, with these values
of $\alpha$ and $\gamma$, the condition $\chi  C_2(\theta,\alpha,\gamma) +[ \chi C_3(\theta,\alpha,\gamma)]^{2(\gamma-1)} < C_1(\alpha,\gamma)$
asymptotically writes   
$\frac{C_0(4\sigma)}{4\pi\sqrt\sigma}\chi<0.5$ for $\theta>0$ small (since
$2(\gamma-1)\to 1$). The choice $\sigma=0.13$ seems to be a good one and we find numerically
$C_0(4\sigma)\simeq 0.6895$, whence the condition $\chi\leq 3.2856$.

\medskip

For (ii), choose $\gamma=1.56$ and $\alpha=0.009$. The result follows from a numerical computation. For any $p>2.6$, it holds that
$\gamma \in (\frac32,\frac{2p+2}{p+2})$.

\medskip

For (iii), choose $\gamma=1.62$ and $\alpha=0.045$. For any $p>3.3$, it holds that
$\gamma \in (\frac32,\frac{2p+2}{p+2})$.

\medskip

For (iv), choose  $\gamma=1.63$ and $\alpha=0.067$. For any $p>3.5$, it holds that
$\gamma \in (\frac32,\frac{2p+2}{p+2})$.

\medskip

For (v), we choose $\gamma=1.63$ and $\alpha=0.08$.  For any $p>3.5$, it holds that
$\gamma \in (\frac32,\frac{2p+2}{p+2})$. We have $C_1(\alpha,\gamma)\simeq 0.502$ and 
$\kappa(\frac12,\gamma-1)\simeq 1.411$ and $\kappa(\gamma-\frac32,\gamma-\frac12)<7.751$. Hence
\begin{gather*}
\chi C_2(\theta,\alpha,\gamma)=\chi \sqrt{\theta}\frac{\sqrt\alpha (\gamma-1) C_0\big(\frac{4\alpha}\theta\big)\kappa\big(\frac12,\gamma-1\big)\kappa\big(\gamma-\frac32,\gamma-1\big)}{2\pi}\simeq 0.286 \,C_0\Big(\frac{4\alpha}\theta\Big)\chi \sqrt{\theta} ,\\
\chi C_3(\theta,\alpha,\gamma)=\chi \sqrt{\theta}\frac{ C_0\big(\frac{4\alpha}\theta\big) \kappa\big(\frac{1}{2}, \gamma-1\big)}{4\pi\sqrt\alpha(4-2\gamma)}\simeq 0.537\, C_0\Big(\frac{4\alpha}\theta\Big)\chi \sqrt{\theta}.
\end{gather*}
Since $\lim_{\theta\to \infty}C_0(\frac{4\alpha}{\theta})=C_0(0)\simeq 0.429$,
the condition $\chi  C_2(\theta,\alpha,\gamma) +[ \chi C_3(\theta,\alpha,\gamma)]^{2(\gamma-1)} < C_1(\alpha,\gamma)$
asymptotically rewrites $0.123\, \chi \sqrt{\theta}
+[0.231 \, \chi \sqrt{\theta}]^{1.26}<0.502$. This holds true if 
$\chi \sqrt{\theta}<1.65$.
\end{proof}

\appendix
\section{Proof of Lemma~\ref{lemma:FI}}
\label{sec:app}

We fix $b>a>0$, introduce $\Gg=\{g:\R_+\to\R_+ : 0 \leq g(s)\leq 1/s$ a.e.$\}$
and, for $g\in\Gg$, $I_a(g)=\int_0^\infty g^{1+a}$ and
$I_b(g)=\int_0^\infty g^{1+b}$. We set $\Gg_b=\{g \in \Gg : 0<I_b(g)<\infty\}$ and
$$
\bar \kappa(a,b)=\sup \{[I_b(g)]^{-\frac a b}I_a(g): g \in \Gg_b \}.
$$
We will show that $\bar \kappa(a,b)=\kappa(a,b)$ and this will complete the proof, because for $f:[0,t]\to \R_+$, the function
$g(s)=\frac{1}{s+f(s)}{\bf 1}_{\{s\in [0,t]\}}$ belongs to $\Gg$.

\medskip

{\it Step 1.} Here we show that $\bar \kappa(a,b)<\infty$ and that
there exists $g \in \Gg_b$ realizing the supremum.
\medskip

First note that for any $g \in \Gg$, by the H\"older inequality,
$$
I_a(g)\le\int_0^1 g^{1+a} + \int_1^\infty \frac{\dd s}{s^{1+a}} \leq [I_b(g)]^{\frac{1+a}{1+b}} + \frac 1 a.
$$

Next we observe that $\bar \kappa(a,b)=\sup \{I_a(g): g\in \Gg_{b,1}\}$, where $\Gg_{b,1}=\{g \in \Gg : I_b(g)=1\}$. This easily follows from the fact that for any $g \in \Gg_b$ and any $\lambda>0$, the function
$g_\lambda(s)= \lambda g(\lambda s)$ still belongs to $\Gg_b$, and 
$I_a(g_\lambda)=\lambda^{a}I_a(g)$ and $I_b(g_\lambda)=\lambda^{b}I_b(g)$.

\medskip

The two above points show that $\bar\kappa(a,b)\leq 1+\frac1a<\infty$. Now we consider
a sequence $(g_n)_{n \geq 1}$ of $\Gg_{b,1}$ such that 
$\lim_n I_a(g_n)=\bar \kappa(a,b)$ and we set $h_n=g_n^{1+a}$. Since $I_b(g_n)=1$, the family $(h_n)_{n\geq 1}$ takes values in the unit ball of $L^{p}(\R_+)$, where $p=\frac{1+b}{1+a}>1$, so that we can find a (non relabeled) subsequence of $(h_n)_{n\geq 1}$ converging weakly in $L^p(\R_+)$ to some function $h$. One easily verifies that $g :=h^{\frac1{1+a}}\in \Gg$, and
it classically holds true that 
$$
I_b(g)=||h||_p^p\leq
\liminf_n ||h_n||_p^p
=\liminf_n I_b(g_n)=1.
$$
We now show that $I_a(g)=\bar\kappa(a,b)$. The weak convergence of
$h_n$ to $h$ implies that for every $\e\in(0,1)$, one has $\int_\e^{1/\e} g_n^{1+a}=\int_\e^{1/\e} h_n \to \int_\e^{1/\e} h=\int_\e^{1/\e}g^{1+a}$. To conclude that $I_a(g)=\lim_nI_a(g_n)=\bar\kappa(a,b)$,
it suffices to note that, by the H\"older inequality and since $g_n \in \Gg_{b,1}$,
$$
\lim_{\e\to 0}\sup_{n\geq 1} \Big(\int_0^\e g_n^{1+a}+\int_{1/\e}^\infty g_n^{1+a}\Big)
\leq \lim_{\e\to 0}\sup_{n\geq 1} \Big([I_b(g_n)]^{\frac{1+a}{1+b}}\e^{\frac{b-a}{1+b}}
+ \int_{1/\e} \frac{\dd s}{s^{1+a}}\Big)=0.
$$
All in all, $g \in \Gg$ and $[I_b(g)]^{-\frac a b}I_a(g)\geq \bar \kappa(a,b)$, 
whence necessarily $[I_b(g)]^{-\frac a b}I_a(g)=\bar \kappa(a,b)$.

\medskip

{\it Step 2.} By Step 1, there is $g \in \Gg_b$ realizing the supremum.
Here we show that there is a constant $k>0$, namely $k=[\frac{b(1+a)I_b(g)}{a(1+b)I_a(g)}]^{\frac{1}{b-a}}$, 
such that for a.e. $s>0$, we have $g(s)=\max\{k,s^{-1}\}$.

\medskip

{\it Step 2.1.} We first show that $g>0$ a.e. It suffices to show that for all $\e\in (0,1)$, $\lambda(A_\e)=0$, where $A_\e=\{s \in [0,\frac1\e] : g(s)=0\}$
and where $\lambda$ is the Lebesgue measure. For any $\alpha\in (0,\e)$, the function $g_\alpha=g+\alpha {\bf 1}_{A_\e}$ belongs to $\Gg_b$.
Hence
$$
\frac{I_a(g)}{[I_b(g)]^{\frac ab}} \geq \frac{I_a(g_\alpha)}{[I_b(g_\alpha)]^{\frac ab}}
=\frac{I_a(g)+\lambda(A_\e)\alpha^{1+a}}{[I_b(g)+\lambda(A_\e)\alpha^{1+b}]^{\frac ab}}=\frac{I_a(g)}{[I_b(g)]^{\frac ab}} \Big(1+ \lambda(A_\e)\Big[\frac{\alpha^{1+a}}{I_a(g)}-\frac ab \frac{\alpha^{1+b}}{I_b(g)}\Big]+O(\alpha^{2+a+b})\Big)
$$
as $\alpha\to 0$. This implies that $\lambda(A_\e)=0$, because
$\alpha^{1+a}\gg\alpha^{1+b}$ as $\alpha\to 0$.

\medskip

{\it Step 2.2.} Now we show that $g\leq k$ a.e. 
By Step 2.1, it suffices to prove that for all $\e\in (0,1)$, we have 
$g\leq k$ a.e. on $B_\e=\{s \in [0,\frac1\e] : g(s)\geq \e\}$.
For all $\delta:\R_+\to\R_+$ such that $\delta\leq {\bf 1}_{B_\e}$, all $\alpha \in (0,\e)$, it holds that $g_{\delta,\alpha}=g-\alpha \delta$ belongs to $\Gg_b$. The function
$\frac \delta g$ is bounded and compactly supported, which allows one to justify the following computation: Since $g_{\delta,\alpha}^{1+a}=g^{1+a}(1-\alpha\frac\delta g)^{1+a}
=g^{1+a}-\alpha (1+a)g^a \delta +O(\alpha^2)$ and $g_{\delta,\alpha}^{1+b}=g^{1+b}-\alpha (1+b) g^b \delta +O(\alpha^2)$ as $\alpha\to 0$,
$$
\frac{I_a(g)}{[I_b(g)]^{\frac ab}}\geq 
 \frac{I_a(g_{\delta,\alpha})}{[I_b(g_{\delta,\alpha})]^{\frac ab}}
=\frac{I_a(g)-(1+a)\alpha \int_0^\infty g^a\delta + O(\alpha^2)}{[I_b(g)-(1+b)\alpha \int_0^\infty g^b\delta + O(\alpha^2)]^{\frac ab}}.
$$
Hence 
$$\frac{I_a(g)}{[I_b(g)]^{\frac ab}}\geq 
\frac{I_a(g)}{[I_b(g)]^{\frac ab}}\Big(
1  - \frac{(1+a)\alpha}{I_a(g)} \int_0^\infty g^a \delta +\frac{a(1+b)\alpha} 
{b I_b(g)} \int_0^\infty g^b \delta + O(\alpha^2)\Big).
$$
This implies that
$$
\int_0^\infty g^b \delta \leq \frac{b(1+a)I_b(g)}{a(1+b)I_a(g)}
\int_0^\infty g^a \delta.
$$
Since this holds true for any measurable $\delta:\R_+\to\R_+$ such that
$\delta\leq {\bf 1}_{B_\e}$, we conclude that $g \leq [\frac{b(1+a)I_b(g)}{a(1+b)I_a(g)}]^{\frac1{b-a}}$
a.e. on $B_\e$, which was our goal.

\medskip

{\it Step 2.3.} We next show that $g\geq k$ a.e. on $C=\{s >0 : g(s)<\frac 1 s\}$. By Step 2.1, it suffices to show that for all $\e\in (0,1)$,
$g\geq k$ a.e. on $C_\e=\{s\in [0,\frac 1 \e] : \e \leq g(s) \leq \frac {1-\e} s\}$. For all $\delta:\R_+\to\R_+$ such that $\delta\leq {\bf 1}_{C_\e}$, all $\alpha \in (0,\e^2)$, it holds that $g_{\delta,\alpha}=g+\alpha \delta$ belongs to $\Gg_b$. The function
$\frac \delta g$ is again bounded and compactly supported
and, proceeding exactly as in Step 2.2, we find that 
$$\frac{I_a(g)}{[I_b(g)]^{\frac ab}}\geq 
\frac{I_a(g)}{[I_b(g)]^{\frac ab}}\Big(
1  + \frac{(1+a)\alpha}{I_a(g)} \int_0^\infty g^a \delta -\frac{a(1+b)\alpha} 
{b I_b(g)} \int_0^\infty g^b \delta + O(\alpha^2)\Big)
$$
as $\alpha \to 0$. This implies that
$$
\int_0^\infty g^b \delta \geq \frac{b(1+a)I_b(g)}{a(1+b)I_a(g)}
\int_0^\infty g^a \delta.
$$
Since this holds true for any measurable $\delta:\R_+\to\R_+$ such that
$\delta\leq {\bf 1}_{C_\e}$, we conclude that $g \geq [\frac{b(1+a)I_b(g)}{a(1+b)I_a(g)}]^{\frac1{b-a}}$
a.e. on $C_\e$.

\medskip

{\it Step 2.4.} We now conclude that $g(s)=\min\{k,s^{-1}\}$ for a.e. $s> 0$. We know from Steps 2.2 and 2.3 that $g\leq k$ a.e. and that
$g \geq k$ a.e. on $C=\{s>0 : g(s)<\frac 1 s\}$. We thus have $g=k$ a.e. on $C$.
Let $D=\R_+\setminus C=\{s>0 : g(s)=\frac 1 s\}$ and $r={\rm ess} \,\inf D$. We claim that $r=\frac 1 k$.

\medskip

We know that for a.e. $\e\in (0,r)$, $r-\e \in C$, so that $\frac 1{r-\e}\geq g(r-\e)=k$. Thus $r \leq \frac 1 k$.
Now consider $r_n \in D$ such that $\lim_n r_n=r$. We have 
$k\geq g(r_n)=\frac 1 {r_n}$, so that $r \geq \frac 1 k$.

\medskip

We have shown that $g=k$ a.e. on $[0,\frac 1 k]\subset C$, and it remains to verify that for a.e. $s>\frac 1 k$, we have $g(s)=\frac 1 s$.
This follows from the fact that if $g(s)<\frac 1 s$ for some $s>\frac1k$, then $s \in C$, so that $g(s)=k$, which is not possible since $k>\frac 1 s$.

\medskip

{\it Step 3.} For any $k>0$, the function $g(s)=\max\{k,s^{-1}\}$
satisfies $I_a(g)=k^{a}+ \frac 1 a k^a$ and $I_b(g)=k^{b}+\frac 1 b k^b$. We deduce from Step 2 that
$$
\bar \kappa(a,b)=\frac{k^{a}+\frac 1 a k^a}{(k^{b}+\frac 1 b k^b)^{a/b}}
=\frac{1+\frac 1a}{(1+\frac 1b)^{a/b}},
$$
which is nothing but $\kappa(a,b)$.
\hfill $\square$

\paragraph{Acknowledgements}
We warmly thank Vincent Calvez and Beno\^{\i}t Perthame for crucial discussions regarding our key functional inequality and its proof. We also thank the referee for
their comments that allowed us to improve the clarity of this paper.

\small

\bibliography{biblio}

\end{document}